\documentclass[11pt,reqno]{amsart}
\usepackage[T1]{fontenc}
\usepackage{verbatim}
\usepackage{prettyref}
\usepackage{mathrsfs,amsmath, amsthm,amssymb,bbm,fullpage}
\usepackage[unicode=true,pdfusetitle,
 bookmarks=true,bookmarksnumbered=false,bookmarksopen=false,
 breaklinks=false,pdfborder={0 0 1},backref=false,colorlinks=false]
 {hyperref}

\makeatletter

\numberwithin{equation}{section}
\numberwithin{figure}{section}
\theoremstyle{plain}
\newtheorem{thm}{\protect\theoremname}[section]
  \theoremstyle{definition}
  \newtheorem{example}[thm]{\protect\examplename}
  \theoremstyle{remark}
  \newtheorem{rem}[thm]{\protect\remarkname}

  \theoremstyle{definition}
  \newtheorem{defn}[thm]{\protect\definitionname}
  \theoremstyle{plain}
  \newtheorem{lem}[thm]{\protect\lemmaname}
  \theoremstyle{plain}
  \newtheorem{cor}[thm]{\protect\corollaryname}
  \theoremstyle{plain}
  \newtheorem{prop}[thm]{\protect\propositionname}

\newcommand{\cA}{\mathcal{A}}
\newcommand{\cC}{\mathcal{C}}
\newcommand{\cD}{\mathcal{D}}
\newcommand{\cH}{\mathcal{H}}

\newcommand{\cE}{\mathcal{E}}
\newcommand{\cF}{\mathcal{F}}

\newcommand{\cL}{\mathcal{L}}

\newcommand{\cX}{\mathcal{X}}

\newcommand{\cP}{\mathcal{P}}
\newcommand{\sP}{\mathscr{P}}
\newcommand{\cQ}{\mathcal{Q}}
\newcommand{\sQ}{\mathscr{Q}}
\newcommand{\cR}{\mathcal{R}}
\newcommand{\cS}{\mathcal{S}}
\newcommand{\R}{\mathbb{R}}

\newcommand{\prob}{\mathbb{P}}
\newcommand{\E}{\mathbb{E}}

\newcommand{\indicator}[1]{\mathbbm{1}_{#1}}

\newcommand{\tensor}{\otimes}

\newcommand{\bsig}{\boldsymbol{\sigma}}

\newcommand{\abs}[1]{\lvert#1\rvert}
\newcommand{\norm}[1]{\lVert#1\rVert}

\newcommand{\supp}{\operatorname{supp}}
\newcommand{\cov}{\operatorname{Cov}}

\newrefformat{prop}{Proposition \ref{#1}}
\newrefformat{cor}{Corollary \ref{#1}}
\newrefformat{rem}{Remark \ref{#1}}
\newrefformat{subsec}{Section \ref{#1}}
\newrefformat{app}{Appendix \ref{#1}}
\newrefformat{def}{Definition \ref{#1}}

\newcommand{\parisicondition}{\textbf{P}}
\newcommand{\NEV}{\textbf{PREV}}
\newcommand{\RSB}{\textbf{RSB}}
\newcommand{\FEB}{\textbf{FEB}}
\newcommand{\repev}{replicon eigenvalue}
\newcommand{\PSD}{\sP}
\newcommand{\Lip}{\operatorname{Lip}}
\newcommand{\assA}{\textbf{A}}
\newcommand{\pmeas}{\mu}
\newcommand{\Var}{\operatorname{Var}}
\newcommand{\eps}{\epsilon}
\renewcommand{\liminf}{\varliminf}
\renewcommand{\limsup}{\varlimsup}

\theoremstyle{definition}
 \newtheorem*{assm*}{Assumption}
\newtheorem{quest}{Question}

\makeatother

  \providecommand{\corollaryname}{Corollary}
  \providecommand{\definitionname}{Definition}
  \providecommand{\examplename}{Example}
  \providecommand{\lemmaname}{Lemma}
  \providecommand{\propositionname}{Proposition}
  \providecommand{\remarkname}{Remark}
\providecommand{\theoremname}{Theorem}

\begin{document}

\title{Spectral Gap estimates in mean field spin glasses}

\author{Gérard Ben Arous}
\address[Gérard Ben Arous]{Courant Institute, New York University}
\email{benarous@cims.nyu.edu}
\author{Aukosh Jagannath}
\address[Aukosh Jagannath]{Department of Mathematics, Harvard University}
\email{aukosh@math.harvard.edu}

\date{\today}
\begin{abstract}
We show that mixing for local, reversible dynamics of  mean field spin glasses is exponentially slow in the low temperature regime.
We introduce a notion of free energy barriers for the overlap, and prove that their existence imply that the spectral gap is exponentially small, and thus that mixing is exponentially slow. We then exhibit sufficient conditions on the equilibrium Gibbs measure which guarantee the existence of these barriers, using the notion of replicon eigenvalue and 2D Guerra Talagrand bounds. We show how these sufficient conditions cover large classes of Ising spin models for reversible nearest-neighbor dynamics and spherical models for Langevin dynamics. 
Finally, in the case of Ising spins, 
Panchenko's recent rigorous calculation \cite{Panch15} of the free energy for a system 
of ``two real replica''
enables us to prove a quenched LDP for the overlap distribution, which gives us a wider criterion for slow mixing directly related to
the Franz-Parisi-Virasoro approach \cite{FPV92, KurParVir93}. This condition holds in a wider range of temperatures.
\end{abstract}
\maketitle

\section{Introduction}

We prove here that local, reversible dynamics for a general class of mean field spin glasses are exponentially slow  
in the low temperature, or Replica Symmetry Breaking (RSB), phase for a broad class of Ising and spherical models. 
 More precisely, we give sufficient conditions for the spectral gap of these dynamics to be exponentially small. 
In the case of Ising spin models, we provide a wider criterion that holds in a broader range of temperatures. 

The study of the convergence to equilibrium for dynamics of mean-field spin glasses has a rich history  
in the physics literature, and it is impossible to give here anything close to an exhaustive description of these results. 
We refer instead to the general surveys \cite{BCKM98, Cug03, BerBir11}. We concentrate here on a basic aspect, the time to equilibrium 
of reversible dynamics for these models should scale exponentially in the system size. This is what we aim to prove. 

The long-time behavior of spin glass dynamics has a very rich phenomenology. 
Along with the time to equilibrium, there is the phenomenon of aging, which 
occurs on timescales that are very long but shorter than the time to equilibrium.
Aging for mean-field spin glasses has been 
extensively studied in the mathematical literature, mostly for a simple class of dynamics, the Random Hopping Time (RHT) 
dynamics. This was done first for the Random Energy Model (REM), see \cite{BABG02,BABG03,BABG03critical}, following 
the seminal works in physics \cite{Bou92,BouDean95}. This was later extended to $p$-spin models
 \cite{BABC08,BAGun12,BovGay13} again for simple RHT dynamics. Understanding aging for Metropolis dynamics
 for general spin glasses is still an important open question, except in the case of the REM, where this
 has been achieved in the recent remarkable works \cite{CerWas17} and \cite{Gay16}. 
 For more on this, see also \cite{BA02, Gui07, Gay16clock,MatMou15}

The mathematics literature related to the spectral gap or for the mixing time at low temperature, however,  is sparser than the one 
related to aging. The behavior of the mixing time has been understood in detail in the simple case of the Random Energy Model 
since the early work \cite {FIKP98}. An upper bound for the mixing time (or for the related notion of thermalization time) is given by 
\cite{Mat00} for Glauber dynamics for hard spin models. Bounds on the correlation time for a single spin were studied in \cite{montanari2006rigorous} 
for Glauber dynamics of dilute $p$-spin models at moderate temperatures. Finally, very recently, \cite{GJ16} gives exponential bounds 
for the spectral gap for the  spherical pure $p$-spin model. 
For the sake of completeness, we mention here that the related question of 
mixing times or spectral gaps for short-range spin glass models goes back at least to \cite{GuiZeg96} and \cite{Des02}.
We  finally note here that the spectral approach to dynamics can also be useful for the study of aging. This was
observed initially in the physics literature by \cite{MelBut97}, 
and detailed in a simple context (the REM-like trap model) by \cite{BovFag05}. 

Let us now explain the core of our approach to proving slow mixing in this work.
In order to understand the long-time dynamical behavior, we study the evolution of the overlap of two replica.
This point of view is naturally inspired by the well-known fact that the order parameter for the study of equilibrium Gibbs measures 
of mean-field spin glasses is the distribution of the overlap of two replica. 
This was the seminal insight of Parisi  \cite{parisi1983order,MPV87} and  
 has been developed in a monumental work by Talagrand \cite{TalBK11, TalBK11Vol2,TalPF}, building on work by Guerra \cite{Guerra03}, 
and much further 
expanded recently by Panchenko \cite{PanchUlt13,Panch15,PanchSKBook}, following the work of Aizenman-Sims-Starr 
\cite{ASS03} and Aizenman-Arguin \cite{ArgAiz09}. For more recent results see also 
\cite{AuffChen13, JagTobPD15, chen2016energy, AufJag16}.

To this end, we introduce ``replicated'' dynamics, i.e., dynamics of two replica evolving independently.
Our aim is to bound the spectral gap of this replicated dynamics. 
Using simple adaptations of classical tools like the Cheeger inequality, 
we first prove that the existence of a free energy barrier for the overlap (to be defined shortly) 
implies that this spectral gap is exponentially small.
 We then use Talagrand's ``2-dimensional Guerra interpolation'' estimates \cite{TalPF},
to provide broad sufficient conditions on the limiting Parisi measure to ensure
that these free energy barriers  exist. 
An important role in the formulation of this sufficient condition is played by what is called the ``Replicon eigenvalue''
which was first introduced by Parisi \cite{Par80} and recently studied in \cite{JagTobPD15, JagTob16boundingRSB}. 
Our results are then shown to cover a broad class of spin glass models at low temperatures, both for Ising spins and spherical 
models. 

In the case of Ising spin models, we introduce a deeper tool which 
is a (quenched) large deviation principle for the overlap distribution. 
This large deviation principle is based on the recent deep results of \cite{Panch15}
which rigorously obtains the free energy for a system 
of ``$m$ real replica'' \cite{FPV92, KurParVir93}. 
We then give a more robust sufficient condition based on the rate function for this large deviation principle
and show how this approach is related to the Replicon eigenvalue. This approach 
applies to a much broader family of models, and even implies exponentially slow mixing
well within the high temperature, or  Replica Symmetric (RS), phase.

It might be worthwhile to compare our approach with the recent work of Gheissari and one of the authors in \cite{GJ16} for the pure $p$-spin spherical 
model. In \cite{GJ16}, the basic tool is the recent understanding of the complexity of the geometry of the random landscape 
at zero temperature introduced in \cite{ABC13} and in \cite{ABA13}, and deepened in \cite{Sub15} and 
\cite{SubZeit16}. This approach has allowed Subag to obtain a very detailed understanding of the Gibbs measure at very 
low temperatures \cite{SubGibbs16} 
along the lines of the Thouless-Anderson-Palmer (TAP) approach, and to show that extensive 
barriers for the Hamiltonian exist in this regime. The recent work \cite{GJ16} 
builds on this fact to show slow mixing for Langevin dynamics. 
Our approach, however,  builds on the existence of free energy barriers for the overlap of two replica, 
rather than on energy barriers for  one replica. 
In the language of physics, \cite{GJ16} uses a ``complexity based'' approach similar to the dynamical TAP approach of \cite{biroli1999dynamical} whereas
this paper uses a ``replica'' approach.
It is thus applicable to cases where a detailed understanding of the energy landscape is lacking
whereas the understanding of the overlap behavior is sharper.

\subsection{Ising Spin Models}

We begin here by describing succinctly our results for the dynamics of Ising mean-field spin glasses.
Let $\Sigma_{N}=\{-1,1\}^{N}$ be the discrete hypercube in dimension
$N$. The \emph{mixed $p$-spin glass }on the hypercube is the Gaussian
process, $(H_{N}(\sigma))$, indexed by $\Sigma_{N}$ with mean and
covariance 
\begin{equation}
\begin{cases}
\E H_{N}(\sigma)=\frac{h}{N}\sum\sigma_{i}\\
\cov(\sigma^{1},\sigma^{2})=N\xi\left(\frac{1}{N}\sum\sigma_{i}^{1}\sigma_{i}^{2}\right)
\end{cases}.\label{eq:mean-covariance}
\end{equation}
Here $\xi(t)=\sum_{p\geq1}\beta_{p}^{2}t^{p}$ is a power series
with positive coefficients, which we call \emph{the model}, and $h\geq0$
is called \emph{the external field}. We assume that $\xi(1+\epsilon)<\infty$
so that $H_{N}$ is well defined in all dimensions. An important quantity
in the following will be the \emph{overlap}, 
\[
R(\sigma^{1},\sigma^{2})=\frac{1}{N}\sum\sigma_{i}^{1}\sigma_{i}^{2},
\]
which we also denote by $R_{12}$. We call $H_{N}$ the \emph{Hamiltonian,
}$\Sigma_{N}$ the \emph{configuration space, }and corresponding to
$H_{N}$ we define the \emph{Gibbs measure, }
\begin{equation}
\pi_{N}(d\sigma)=\frac{e^{-H(\sigma)}}{Z_{N}}d\sigma,\label{eq:Gibbs}
\end{equation}
where $d\sigma$ is the uniform measure on $\Sigma_{N}$. We call
such models \emph{Ising spin models.}

We now turn to the class of dynamics that we will consider in this paper.
We study nearest neighbor dynamics that are reversible with respect
to the Gibbs measure, $\pi_N$, both in discrete and continuous time. 
More precisely,
let $Q(\sigma^{1},\sigma^{2})$ be a (random) Markov
transition matrix on $\Sigma_N$.
We assume that $Q$ satisfies detailed balance with respect to $\pi_N$:
$\pi_N(\sigma^1)Q(\sigma^1,\sigma^2)=\pi(\sigma^2)Q(\sigma^2,\sigma^1)$
for all $\sigma^1,\sigma^2\in\Sigma_N$, and that
 $Q$ is \emph{nearest-neighbor}, i.e., $Q(x,y)=0$ if $x$ and $y$ differ in more than one coordinate.

Since these results hold for rather general dynamics, let us consider an example of dynamics to which these
results will apply.
\begin{example}
Let $P_{SRW}$ denote the transition kernel for the simple random
walk (SRW) on $\Sigma_{N}$,
\[
P_{SRW}(\sigma^{1},\sigma^{2})=\frac{1}{N}\indicator{\exists i\in[N]:\sigma^{1}(i)\neq\sigma^{2}(i)\text{ and }\sigma^{1}(j)=\sigma^{2}(j)\text{ for all }j\in[N]\backslash\{i\}}.
\]
Consider the transition matrix
\begin{equation}
Q^{Met,SRW}(\sigma^{1},\sigma^{2})=\begin{cases}
P_{SRW}(\sigma^{1},\sigma^{2})\left(1\wedge\frac{\pi_{N}(\sigma^{2})}{\pi_{N}(\sigma^{1})}\right) & \sigma^{1}\neq\sigma^{2}\\
1-\sum_{\sigma\neq\sigma^{1}}P_{SRW}(\sigma^{1},\sigma)\left(1\wedge\frac{\pi_{N}(\sigma)}{\pi_{N}(\sigma^{1})}\right) & \sigma^{1}=\sigma^{2}
\end{cases}.\label{eq:metropolis-regular}
\end{equation}
We then consider the following two processes.
 The \emph{discrete time
Metropolis chain with base chain the SRW} is the Markov chain, $(\sigma_{d}(n))_{n\geq1}$,
with transition matrix $Q^{Met,SRW}$. 
The \emph{continuous time Metropolis
chain with base chain the SRW} is the continuous time Markov process,
$(\sigma_{c}(t))$, on $\Sigma_{N}$ with infinitesimal generator
$I-Q^{Met,SRW}$.
\end{example}

The \emph{spectral gap }of $I-Q$, call it $\lambda_{1}$, is the
first non-trivial eigenvalue of $I-Q$. 
Our goal is to prove that in a certain regime, called the \emph{spin
glass }or \emph{Replica Symmetry Breaking (RSB) phase, }$\lambda_{1}$
will be exponentially small, so that the corresponding induced dynamics
will mix slowly. (For a brief reminder how spectral gap estimates
relate to mixing see \prettyref{sec:Examples-and-Applications}.)

To this end, we begin by introducing the (static) notion of free energy barriers for the overlap. 
In \prettyref{sec:The-difficulty-and-spectral-gap}, we introduce a much broader notion of the difficulty of the landscape of the Gibbs 
measure and use this to bound the spectral gap. 
For the sake of exposition in this introduction, however, we restrict it here to the following simpler notion. 
Heuristically, a number $q_2$  is a free energy barrier for the overlap if the probability that two replica have  
an overlap close to $q_2$ is exponentially small, whereas two other values $q_1$ and $q_3$ 
on either side of $q_2$ are probable for the overlap. More precisely, we have the following.

\begin{defn}
\label{def:FEB} We say that there exists a \emph{free energy barrier} of height $C>0$ for the overlap if there exists a triple  $-1\leq q_1<q_2<q_3\leq 1$, $0<\epsilon<\frac{1}{4} \min\{q_3-q_2,q_2-q_1\}$ 
such that
\begin{equation}\label{eq:exponential-rarity}
\limsup_{N\to\infty}\frac{1}{N}\log\prob\left(\pi_{N}^{\tensor2}\left(R_{12}\in (q_2-\epsilon,q_2+\epsilon)\right) > e^{-CN}\right)<0
\end{equation} 
and such that, for $i=1$ and $i=3$
\begin{equation}\label{eq:separation}
\liminf_{N\to\infty}\E\pi_{N}^{\tensor2}\left(R_{12}\in (q_i-\epsilon,q_i+\epsilon)\right) >0
\end{equation}
If there is a free energy barrier of height $C$ for some $C>0$, then we say that \FEB~holds.
\end{defn}

Our first result is then that the existence of a free energy barrier for the overlap implies that the spectral gap is exponentially small.
\begin{thm}
\label{thm:hard-spin-FEB-thm} If there exists a free energy barrier for the overlap of height $C>0$, then
\begin{equation}
\limsup_{N\to\infty}\frac{1}{N}\log\prob(\frac{1}{N}\log\lambda_{1}>-C)<0.
\end{equation}
\end{thm}
We will provide a stronger bound that generalizes the above shortly. Said bound, however, will 
have a more limited range of applicability for technical reasons. For the sake of exposition, we postpone this to \prettyref{sec:LDP}.

The next step is naturally to find good sufficient conditions to insure the existence of free energy barriers for the overlap.
In order to do, consider the
notion of a \emph{limiting overlap distribution}. Let
\begin{equation}
\zeta_{N}(\cdot)=\E\pi_{N}^{\tensor2}\left(R_{12}\in\cdot\right).\label{eq:zeta-N}
\end{equation}
Since $-1\leq R_{12}\leq1$, the sequence $(\zeta_{N})$ is tight.
A \emph{limiting overlap distribution }is any weak limit point of
this sequence 

\begin{equation}
\lim_{N\to\infty}\E\pi_{N}^{\tensor2}\left(R_{12}\in\cdot\right)=\zeta.\label{eq:overlap-distn}
\end{equation}
In the following, it is convenient to make the following technical
assumption.

\begin{assm*}\assA\label{ass:A} 
	There is a unique limiting overlap distribution, $\zeta$.
\end{assm*}

\noindent It is known that this assumption holds in a large class
of models. For more on this see \prettyref{sec:Examples-and-Applications}.
For the remainder of this paper we will assume \assA, and we will
refer to $\zeta$ simply as \emph{the} \emph{limiting overlap distribution.}

One is most interested in the properties of the support of $\zeta$ particularly its topology. When $\zeta$ is an atom,
$(\xi,h)$ is said to be in the \emph{Replica Symmetric (RS) }phase.
When $\zeta$ is not an atom, $(\xi,h)$ is said to be in the \emph{Replica
Symmetry Breaking (RSB) }phase\emph{. }In the language of statistical
physics, one thinks of the RS phase as corresponding to classical
high temperature behavior and RSB as spin glass, or low temperature,
behavior. 

We will work in the following regime throughout this paper.
\begin{defn}
\label{def:RSB} We say that\textbf{ }\RSB~ holds if \assA~ holds
and $\zeta$ is not an atom.
\end{defn}
\noindent With this terminology, we can now restate our goals. We
aim to provide an analytical criterion regarding the support of $\zeta$
and the pair $(\xi,h)$ when \RSB~ holds that will imply that $\lambda_{1}$
decays exponentially in $N$. To this end, we introduce the following
analytical tools from the study of spin glasses. 

For $\nu\in\Pr([0,1])$, the\emph{ Parisi functional}, $P_{I}(\nu)$,
is 
\begin{equation}
P_{I}(\nu)=\phi_{\nu}(0,h)-\frac{1}{2}\int_{0}^{1}\xi''(s)s\nu[0,s]ds\label{eq:P-func}
\end{equation}
where $\phi_{\nu}$ is the unique weak solution of 
\begin{equation}
\begin{cases}
\partial_{t}\phi_{\nu}+\frac{\xi''}{2}\left(\Delta\phi_{\nu}+\nu([0,s])(\partial_x\phi_{\nu})^{2}\right)=0\\
\phi_{\nu}(1,x)=\log\cosh(x)
\end{cases}\label{eq:PPDE-IVP-ising}
\end{equation}
(For the definition of weak solution and basic properties of $\phi_{\nu}$
see \prettyref{app:Analytical-Properties-of} or \cite{JagTobSC15}.)
It is known that $P_{I}$ is continuous and strictly convex \cite{AuffChen14},
and in particular has a unique minimizer. The Parisi functional provides
a variational formula for what is called the free energy:
\begin{equation}
F=\lim\frac{1}{N}\log\int e^{-H(\sigma)}d\sigma=\min_{\nu\in\Pr([0,1])}P_{I}(\nu),\label{eq:parisi-formula-ising-minimizing}
\end{equation}
where $d\sigma$ is the uniform measure on $\Sigma_{N}$. This formula,
called the \emph{Parisi formula}, was proved by Talagrand \cite{TalPF}
for even $\xi$ and Panchenko \cite{PanchPF14} for 
general $\xi$. The minimizer of this problem will play an important
role in our analysis.
\begin{defn}\label{def:parisi-measure}
The \emph{Parisi measure }is the minimizer of \prettyref{eq:parisi-formula-ising-minimizing},
which we denote by $\pmeas$.
\end{defn}
We now turn to defining the main analytical quantity of interest,
the \repev. For $\nu\in\Pr([0,1])$, consider the solution of the
SDE 
\begin{equation}
dX_{t}=\xi''(s)\nu(s)\partial_x \phi_{\nu}(s,X_{s})ds+\sqrt{\xi''(s)}dW_{s}\label{eq:local-fields}
\end{equation}
with initial data $X_{0}=h$, where $W_{s}$ is a standard Brownian
motion. Following the physics literature \cite{MPV87}, we will refer
to $X_{t}$ as \emph{the local field process.} We note here that $\partial_x \phi_{\nu}$
is continuous in time and smooth and bounded in space (see \prettyref{app:Analytical-Properties-of}
below or \cite{JagTobSC15}) so this solution exists in the It\^o
sense. 

For $\nu\in\Pr([0,1])$ and $q\in\supp(\nu)$ , the \emph{\repev}
is
\begin{equation}
\Lambda_{R}(q,\nu)=1-\xi''(q)\E_{h}\left(\Delta \phi_{\nu}\right)^{2}(q,X_{q}).\label{eq:replicon-eigenvalue}
\end{equation}
Here and in the following we denote the support of  a probability
measure, $\nu$, by $\supp(\nu)$. We can now define the main analytical
condition for our results. 
\begin{defn}
\label{def:NREV} A pair $(\xi,h)$ is said to satisfy \NEV, if:
\begin{itemize}
\item \textbf{RSB }holds. 
\item There are at least two points in $E=\supp(\zeta)\cap\supp(\pmeas).$
\item There is a $q\in E$ with a positive \repev:
\[
\Lambda_{R}(q,\pmeas)>0.
\]
\end{itemize}
\end{defn}
\begin{rem}
The condition \NEV~ is related to a generalization of the de Almeida-Thouless
line \cite{AT78}. More precisely, it can be shown \cite{AuffChen13,JagTobPD15}
that if the \repev~is positive for some $q\in\supp(\pmeas)$,
then $q$ is an isolated element of $\supp(\pmeas).$ As we will point
out in \prettyref{thm:exp-rare}, the condition that $\Lambda_{R}(q,\pmeas)>0$
implies that $q$ is an isolated element of the support of $\zeta$
when \NEV~ holds. 
\end{rem}
With these notions in hand, we then have the following theorem.
\begin{thm}
\label{thm:hard-spin-main-thm} If $\xi$ is convex and $(\xi,h)$ satisfy \NEV, then \FEB~ holds. In particular,
for some $C>0$,
\[
\limsup_{N\to\infty}\frac{1}{N}\log\prob(\frac{1}{N}\log\lambda_{1}>-C)<0.
\]
\end{thm}

Let us now turn to an example of models to which our result applies.
One class of models to which our results hold are the
following. We say that a model is \emph{generic} if the family
of monomials
\[
\cF = \{t^{p}:\beta_{p}\neq0\}\cup\{1\}
\]
is total in $(C([-1,1]),\sup\abs{\cdot})$. We say that a model is \emph{even generic}
if $\cF$ is total in $C([0,1],\sup\abs{\cdot})$ and $\xi$ is even. (For instance, the 
latter case holds
if $\sum_{p: \beta_p\neq0}p^{-1} =\infty$ by the M\"untz-Szaz theorem.)
\begin{thm}
\label{thm:ising-spin-unif-elliptic} 
Suppose that $\xi=\beta^{2}\xi_{0}$ has $\xi''_0(0)=0$ and is convex and either generic or even generic. 
Then there
is an $h_0(\beta,\xi)>0$ such that for $h\leq h_0$,
if the Parisi measure, $\pmeas$, is not an atom, then \NEV~
holds. Consequently,
\[
\limsup_{N\to\infty}\frac{1}{N}\log\prob\left(\frac{1}{N}\log\lambda_{1}>-c\right)< 0
\]
for some $c>0$. In particular, this holds if $\beta$ is sufficiently large.
\end{thm}

As we have shown that the order of decay of $\lambda_1$ is at least
exponential, it is natural to ask if this is indeed the correct order of growth.
Of course, simply knowing that the dynamics are local is
insufficient to determine this question; one needs more assumptions.
A natural assumption is \emph{coercivity} with respect to some base
Markov process.
\begin{defn}
Let $P$ be a transition matrix on $\Sigma_{N}$ that satisfies detailed
balance with respect to the uniform measure. A transition matrix
$Q$ for some reversible Markov chain on $\Sigma_{N}$ is said to
be \emph{P-coercive} if there is some constant $A$ such that 
\[
AP(\sigma^{1},\sigma^{2})\leq Q(\sigma^{1},\sigma^{2}).
\]
\end{defn}
This is simply asserting a form of coercivity between the corresponding
Dirichlet forms (see \prettyref{subsec:Lower-bound-by-stability-discrete}
below). With this definition in hand, we then have the following
\begin{thm}
\label{thm:coercive-mixing} If $Q_{N}$ is  $P_{SRW}$-coercive
with constant $A_{N}$ then there is some constant $c'>0$ such that
\[
\limsup_{N\to\infty}\frac{1}{N}\log\prob(\frac{1}{N}\log\lambda_{1}<-c'+\frac{1}{N}\log A_{N})<0
\]
almost surely for all $N$.
\end{thm}

As an example, note the following
\begin{cor}\label{cor:met-coercive}
$Q^{M,SRW}$ is $P_{SRW}-$coercive. In particular, there is 
a constant $c>0$, such that
\[
\limsup_{N\to\infty}\frac{1}{N}\log\prob(\frac{1}{N}\log \lambda_1 <-c)<0
\]
for all $N$ almost surely.
\end{cor}
It is known that the latter limit exists almost surely and is given
by a constant \cite{GuerraTon02}. For a variational representation
of this constant in this setting see \cite{auffinger2016parisi}.

This bound is rather coarse and makes no mention
of either the temperature or the overlap distribution. As an application 
of the deep study by Mathieu \cite{Mat00}, 
it can be shown that if we study $\xi$ of the form $\xi=\beta^2\xi_0$,
then the thermalization time (which is related to, but slight
different from the mixing time)
can be bounded in terms of the free energy $F(\beta)$. 
In particular, it is bounded by $F'(\beta)$, which is known \cite{Panch08}
to satisfy
\[
F'(\beta)=\beta \int \xi_0(1)-\xi_0(t)d\mu.
\]

\subsection{A better bound for the difficulty and a Large Deviation Principle for the Overlap in Ising Spin models}\label{sec:LDP}
Let us now explain a more general principle than the one described above. 
This provides a sharper result that  has a range of 
validity beyond that described above. 
It will hopefully also clarify the role of the condition \NEV~ in the preceding discussion.
At the moment, however, these results will only apply to Ising spin models, 
whereas the discussion above extends to spherical models.

We  begin by proving a quenched large deviation principle for the overlap distribution.
To this end, consider the sequence of measures 
\[
\sQ_{N}(\cdot)=\pi_{N}^{\tensor2}(R_{12}\in\cdot).
\]
(Observe that  $\zeta_N=\E \sQ_N$.) With this in hand, we 
have the following theorem.
\begin{thm}
\label{thm:LDP} Let $\xi$ be convex. 
The sequence $\{\sQ_{N}\}$ satisfies an large deviation principle with a rate function $I$
and rate $N$ almost surely. 
\end{thm}
This large deviation principle is a consequence of a recent deep result of Panchenko \cite{Panch15}
which proves the sharpness of  the Guerra-Talagrand bounds. Indeed, the rate function here is given by 
\begin{equation}\label{eq:rate-function}
I(q)=-\inf_{\lambda\in\R,Q\in\cQ_{q},\nu\in\Pr([0,1])}P(\nu,Q,\lambda)+2\min_{\nu}P_{I}(\nu)
\end{equation}
where $P_{I}$ is from \prettyref{eq:P-func} and the set $\cQ_q$ and the functional $P$ 
are defined in \prettyref{sec:2d-gt-ising}. For a related result, see \cite{bovier2009aizenman}. 
In the physics literature, this rate function is referred to as the free energy
for a system for ``two real replica'' \cite{FPV92, KurParVir93}.

With this in hand, we may easily improve upon the spectral bound provided in \prettyref{thm:hard-spin-FEB-thm}.
through \FEB. To this end, define the following quantity
\begin{equation}\label{eq:feb-H}
\cH = \sup_{q_1<q_2<q_3} -I(q_1)+I(q_2)-I(q_3).
\end{equation}
Heuristically, $\cH$ encodes the ``length'' of a passage between $q_1$ and $q_3$ though 
the point $q_2$. 
We then have the following generalization of a free energy barrier.
\begin{defn}
We say that Generalized \FEB~ holds if
$\cH>0.$
\end{defn}
Generalized \FEB~ is of course natural related to \FEB. Indeed, the following
holds.
\begin{prop}\label{prop:generalized-feb-to-feb}
\FEB~implies Generalized \FEB
\end{prop}
We introduce the notion of Generalized \FEB~as it may hold in broader generality. 
Furthermore, it implies the following stronger result.
\begin{thm}\label{thm:generalized-FEB-EVbound}
We have that
\begin{equation}
\limsup\frac{1}{N}\log \lambda_1 \leq - \cH.
\end{equation}
almost surely. In particular, if $\xi$ is convex and Generalized \FEB~holds, this is strictly negative.
\end{thm}

The question of slow mixing then reduces to showing that Generalized \FEB~holds.
To this end, we note the following equivalent statement. Recall that by \cite[Theorem 1]{Panch15},
there is a $q$ such that $I(q)=0$. (In fact, this holds for every $q$ in the support of any limit point of $\zeta$, 
see \prettyref{lem:witness-lemma-RF}.) Thus we see that Generalized \FEB~is equivalent to
the following property of the rate function. 
\begin{prop}\label{prop:equivalence-GFEB}
Generalized \FEB~ holds if and only if there is some $q_0$ with $I(q_0)=0$
such that $I$ is not both: 
\begin{itemize}
\item  non-increasing on $[-1,q_0]$, and
\item  non-decreasing on $[q_0,1]$.
\end{itemize}
\end{prop}

It is now a good time to relate this approach the the one of the previous section.
 To show that $I$ is not monotone, 
we find two points, $q_1,q_3$ that satisfy \prettyref{eq:separation}. By an elementary argument, see
\prettyref{lem:witness-lemma-RF}, this implies that $I(q_i)=0$.  The main observation is that if one of these points, say $q_1$, 
has a positive \repev, then $I>0$ in a neighborhood of that point.
\begin{thm}\label{thm:NREV-and-rate-function}
Suppose that $q_*\in\supp(\mu)$ has a positive \repev, $\Lambda_R(q_*,\mu)>0$. Then there
is a punctured neighborhood of $q_*$,   $E=(q_{*}-\epsilon_{0},q_{*}+\epsilon_{0})\cap(0,1)\setminus\{q_*\}$,
such that for $q\in E$,
\[
I(q)>0
\]
\end{thm}
\noindent Consequently, $I$ will not be monotone.

In this setting, there is a natural analogue of \prettyref{thm:hard-spin-main-thm}. To this
end, we introduce the following definitions which are modifications
of the previous conditions.
\begin{defn}
We say that \textbf{GRSB }holds if the Parisi measure, $\mu$, has
at least two points in its support.
\end{defn}
Observe that this is different from \textbf{RSB }as the former requires
assumptions regarding the overlap distribution. Secondly, we have
a modification of\textbf{ PREV.}
\begin{defn}
We say that \textbf{GPREV} holds if for some
$q\in\supp(\mu),$ 
\[
\Lambda_{R}(q,\mu)>0.
\]
\end{defn}
We then have the following theorem.
\begin{thm}\label{thm:gfeb-main-thm}
Suppose that $\xi$ is convex. If \textbf{GRSB }and \textbf{GPREV} hold, then \textbf{GFEB }holds.
\end{thm}

As an example, of models to which this applies we note that this of course subsumes \prettyref{thm:ising-spin-unif-elliptic},
without the requirement that the model be generic. 

\begin{thm}\label{thm:gfeb-examps}
Suppose that $\xi=\beta^2\xi_0$ has $\xi_0''(0)=0$ and is convex. Then there is an $h_0$ such that for $h\leq h_0$,
if the Parisi measure is not an atom, then \textbf{G\NEV} and \textbf{G\RSB} hold. Consequently,
\[
\limsup_{N\to\infty}\frac{1}{N}\log \prob\left(\frac{1}{N}\log\lambda_1 >-c\right) <0
\]
for some $c>0$. In particular, this holds if $\beta$ is sufficiently large.
\end{thm}

It is natural to ask if the condition \textbf{GFEB} holds in models even when \textbf{G\RSB} does not hold,
i.e., even in the replica symmetric phase. This is discussed presently.\\

\noindent\textbf{Dynamical Phase transitions in Ising spin models.}
We end our discussion of Ising spin models by observing the following.
An important consequence of these results is that they resolve a natural
question raised in the physics literature namely if the \emph{static
glass transition} is always below \emph{the dynamical glass transition}.
To make this precise, let $\nu_{\beta}$ denote the minimizer of \prettyref{eq:parisi-formula-ising-minimizing}
for $\xi$ of the form $\xi=\beta^{2}\xi_{0}$. Define
\begin{align}
\beta_{s} & =\max\left\{ \beta>0:\nu_{\beta}=\delta_{q}\text{ for some }q\in[0,1]\right\} \\
\beta_{d} & =\min\{\beta>0:\lim\prob(-\frac{1}{N}\log\lambda_{1}(\beta)\geq C)=1\text{ for some }C>0\}.\label{eq:beta-d}
\end{align}
It is predicted that $\beta_{s}\geq\beta_{d}$ \cite{CastCav05,MPV87}.
This is a consequence of \prettyref{thm:ising-spin-unif-elliptic}.
In fact, we may go further. If we let
\[
\beta_{GFEB} = \min\{\beta>0:\text{ \textbf{GFEB} holds}\}
\]
then as a consequence of the above, we have 
\begin{cor}\label{cor:Ts-T2}
If $h=0$, $\xi''_0(0)=0$ and $\xi_0$ is convex, then 
\[
\beta_d\leq\beta_{GFEB}<\beta_s
\]
\end{cor}
We end this section with the following natural question.
\begin{quest}
Is it true that $\beta_d=\beta_{GFEB}$?
\end{quest}

\subsection{Spherical Models}

Let us now consider \emph{spherical mixed $p$-spin glasses}, which
we refer to as \emph{spherical models} for short. For these
models the configuration space will be $\cS_{N}=S^{N-1}(\sqrt{N})\subset\R^{N}$,
which we equip with the usual, induced metric, $g,$ and the normalized
volume measure, $d\sigma$. Let $\xi$ and $h$ be as before. The
Hamiltonian for this model will be the Gaussian process on $\cS_{N}$
with mean and covariance given by \prettyref{eq:mean-covariance},
and the Gibbs measure, $\pi_{N}$, will be as in \prettyref{eq:Gibbs}
with respect to the normalized volume measure as opposed to the uniform
measure. 

Our goal in this section is to understand the relaxation time of the
\emph{Langevin dynamics} of this model. The Langevin dynamics is the
heat flow induced by 
\begin{equation}
\cL_{H}=-\Delta+g(\nabla H,\nabla\cdot).\label{eq:Langevin}
\end{equation}
It is known that $H$ is (a.s.) smooth and Morse, so that $\cL_{H}$
is essentially self-adjoint with pure point spectrum. In particular,
this dynamics is uniquely defined. We wish to analyze the asymptotics
of $\lambda_{1}$, the first nontrivial eigenvalue of $\cL_{H}$. 

Again our starting point is by relating the spectral gap to free energy barriers
for the overlap. Define free energy barriers as in \prettyref{def:FEB}. 
We then have the following. 
\begin{thm}
\label{thm:spherical-FEB-thm} If there exists a free energy barrier for the overlap of height $C>0$, then
\begin{equation}
\limsup_{N\to\infty}\frac{1}{N}\log\prob(\frac{1}{N}\log\lambda_{1}>-C)<0.
\end{equation}
\end{thm}

As in the Ising spin setting, we wish to show that \FEB~holds under an \NEV-type 
condition. To this end, let 
\[
\zeta_{N}(\cdot)=\E\pi_{N}^{\tensor2}(R_{12}\in\cdot),
\]
as in \prettyref{eq:zeta-N} except here $\pi_{N}$ is the Gibbs measure
for the spherical model. We may then define $\zeta$, \assA, and
\RSB~ from above analogously. It remains to define the analogue of
\NEV, specifically the \repev .

To this end, consider the \emph{Crisanti-Sommers functional}. For
$\nu\in\Pr([0,1])$, the Crisanti-Sommers functional, $\cC(\nu)$, 
is given by 
\begin{equation}
\cC(\nu)=\frac{1}{2}\left(\int\xi''(s)\varphi_{\nu}(s)+\int\frac{1}{\varphi_{\nu}(s)}-\frac{1}{1-s}ds+h^{2}\varphi_{\nu}(0)\right),\label{eq:cs-func-1}
\end{equation}
where 
\begin{equation}
\varphi_{\nu}(s)=\int_{s}^{1}\nu([0,s])ds.\label{eq:cs-param}
\end{equation}
Observe that $\cC$ is lower semicontinuous and strictly convex, so
that the existence and uniqueness of this minimizer are guaranteed.
The Crisanti-Sommers functional provides a variational formula for
the free energy: 

\begin{equation}
F=\lim\frac{1}{N}\log\int e^{-H}d\sigma=\min_{\nu}\cC(\nu).\label{eq:cs-formula}
\end{equation}
where $d\sigma$ is the normalized volume measure on the sphere. This
formula, called the \emph{Crisanti-Sommers formula,} was proved by
Talagrand \cite{TalSphPF06} in our setting and Chen  \cite{ChenSph13}
for more general $\xi$.

For every $\nu\in\Pr([0,1])$ and $q\in\supp(\nu)$, the \emph{\repev}~ 
for spherical models is
\[
\Lambda_{R}(s,\nu)=\frac{1}{\varphi_{\nu}^{2}(s)}-\xi''(s),
\]
and is related to the case of optimality in a certain obstacle problem
\cite{JagTob16boundingRSB,JagTobLT16}. With this definition, we may
then define the \NEV~ condition analogously to \prettyref{def:NREV}. 

Our main result in this setting is the following.
\begin{thm}
\label{thm:spherical-main-thm}If $\xi$ is convex and $(\xi,h)$ satisfy \NEV, then \FEB~holds. In particular, for
some $C>0$, 
\[
\limsup_{N\to\infty}\frac{1}{N}\log\prob(\frac{1}{N}\log\lambda_{1}>-C)<0.
\]
\end{thm}

Let us now turn to an example of models for which these results apply.
\begin{thm}
\label{thm:spherical-unif-elliptic} Suppose that either: 
\begin{enumerate}
\item $\xi=\beta^{2}\xi_{0}$ is convex and has $\xi''_0(0)=0$ and is  generic or even generic, or
\item  $\xi(t)=\beta^{2}t^{p}$ for even $p\geq4$. 
\end{enumerate}
Then there is an $h_0>0$ such that for $h\leq h_0$, if the Parisi measure, $\pmeas$, is not an atom, then \NEV~holds.
Consequently,
\[
\limsup_{N\to\infty}\frac{1}{N}\log\prob\left(\frac{1}{N}\log\lambda_{1}>-c\right)<0
\]
for some $c>0$. In particular, this holds if $\beta$ is sufficiently large.
\end{thm}

\begin{rem}
We note here that the main result does not use the form of $\cL_{H}$
in an essential way. For example, if $L$ is the infinitesimal generator
for any other reversible dynamics for $\pi_{N}$, then the result
still holds provided the corresponding Carré du champ operator, $\Gamma_{1}(f)(x)$,
satisfies the gradient estimate $\Gamma_{1}(f)(x)\leq Cg(Df,Df)$
for some $C=C(N)$ that grows at most polynomially in $N$. See \prettyref{rem:gamma-1-grad-bound}
\end{rem}

In our setting, the matching exponential lower bound has been proved in \cite{GJ16}. 
The proof provided there works for all $\xi$, though it is stated only for Pure $p$-spin models.
\begin{thm}\label{thm:sg-lower-bound-spherical}
There is a constant $c(\xi,h)$ such that 
\begin{equation}\label{eq:sg-lower-bound-spherical}
\limsup_{N\to\infty}\frac{1}{N}\log\prob(\frac{1}{N}\log\lambda_{1}<-c)<0,
\end{equation}
\end{thm}

\noindent\textbf{Dynamical phase transitions for spherical models.}
An important consequence of these results is that they resolve a natural
question raised in the physics literature namely if the \emph{static
glass transition} is always below \emph{the dynamical glass transition}.
To make this precise, let $\nu_{\beta}$ denote the minimizer of 
 \prettyref{eq:cs-func-1}, for $\xi$ of the form $\xi=\beta^{2}\xi_{0}$.
Define
\begin{align*}
\beta_{s} & =\max\left\{ \beta>0:\nu_{\beta}=\delta_{q}\text{ for some }q\in[0,1]\right\} \\
\beta_{d} & =\min\{\beta>0:\lim\prob(-\frac{1}{N}\log\lambda_{1}(\beta)\geq C)=1\text{ for some }C>0\}.
\end{align*}
It is predicted that $\beta_{s}\geq\beta_{d}$ \cite{CastCav05,MPV87}.
This is a consequence of 
\prettyref{thm:spherical-unif-elliptic}.
\begin{cor}
If $h=0$, $\xi_0(0)=0$, and $\xi_{0}$ is convex and  generic or even generic,
or $\xi_{0}=t^{p}$ for some even $p\geq 4$, then $\beta_{s}\geq\beta_{d}$.
\end{cor}

\subsection{Outline of Paper}
The paper is organized as follows. 
We begin by introducing, in \prettyref{sec:The-difficulty-and-spectral-gap}, the needed bounds on spectral gap, 
along the classical lines of Cheeger inequalities. We begin in \prettyref{sec:general-difficulty-def}, 
by introducing the notion of the landscape difficulty of a function, in a very general context for reversible 
dynamics on metric measure spaces, which cover the cases needed here of weighted graphs 
or Riemannian manifolds. This notion quantifies how long it takes for a given function 
to be ``equilibrated'' by the dynamics, in a exponential time scale. 
We then define the landscape difficulty of a metric measure space as the maximal landscape difficulty, among Lipschitz functions. 
In \prettyref{sec:graphs-difficulty}, we show how this notion of landscape difficulty can be applied 
to the case of finite graphs, and get upper bounds on the spectral gap in \prettyref{sec:upper-bound-diff-graph}. 
We also give a short treatment to lower bounds on the spectral gap through Poincar\'e inequalities 
in \prettyref{subsec:Lower-bound-by-stability-discrete}. In \prettyref{sec:manifolds-difficulty}, 
we give the analogous bounds for the spectral gap for compact Riemannian manifolds.

\prettyref{sec:spectral-gap-and-FEB} is devoted to proving that free energy barriers for mean-field spin glasses 
imply an exponential bound for the spectral gap for their dynamics 
using the abstract tools introduced in \prettyref{sec:The-difficulty-and-spectral-gap}. 
The main idea is  here is that the overlap is a difficult function to equilibrate 
for ``replicated'' dynamics. We do this first in \prettyref{sec:difficulty-FEB-ising} for Ising Spin glasses. 
We show there that the existence of free energy barriers for the overlap imply an exponential upper bound 
for the spectral gap of Ising Spin glasses. We complete the analogous task for spherical spin glasses 
in \prettyref{sec:difficulty-FEB-ising}. In the brief \prettyref{subsec:Spectral-Gap-lower}, 
we prove lower bounds on the spectral gap for both Ising and spherical spin glasses.

\prettyref{sec:Exponential-rarity-ising-spin} is devoted to the proof of existence 
of free energy barriers for the overlap at low temperature, for Ising spin glasses. 
This section only deals with equilibrium quantities. 
We show how the behavior of the overlap distribution, for large N,  encodes 
the existence of free energy barriers, and thus show that the overlap is difficult to equilibrate. 
This section uses recent deep tools about the static behavior of spin glasses, 
like the 2D Guerra-Talagrand bounds, which we recall in \prettyref{sec:2d-gt-ising}. 
In \prettyref{sec:nrev-gtbound-ising}, we show how a positive \repev 
 can help bound the 2D Guerra-Talagrand functional. 
In \prettyref{sec:exp-rare-ising-pf} and \ref{sec:hard-spin-main-thm-proof},
 we show how this information allows us to prove quickly inequalities like \prettyref{eq:exponential-rarity}
 as well as \prettyref{thm:hard-spin-main-thm}.

\prettyref{sec:Overlap-Bounds-Sph} extends the results of \prettyref{sec:Exponential-rarity-ising-spin} 
to the case of spherical spin glasses, using the Crisanti-Sommers formula 
and the spherical version of the 2D Guerra-Talagrand bounds.

In \prettyref{sec:Examples-and-Applications}, we show that wide classes of models satisfy 
our sufficient conditions for slow mixing.

In \prettyref{sec:LDP-proofs}, we improve upon our spectral approach by introducing a generalized \FEB. In 
the process, we will prove an LDP for the overlap distribution which follows from a recent deep result
of Panchenko \cite{Panch15} on the matching lower bound to the 2D~ Guerra-Talagrand bound.

We have pushed some of the most technical properties of the Parisi PDE to \prettyref{app:Analytical-Properties-of} 
(resp. \prettyref{app:The-Spherical-Parisi}) needed in \prettyref{sec:Exponential-rarity-ising-spin} 
(resp. \prettyref{sec:Overlap-Bounds-Sph}).

\subsection*{Acknowledgements}
The authors would like to thank G. Biroli, C. Cammarota, and R. Gheissari for many helpful discussions. 
The authors would also like to thank D. Panchenko and an anonymous referee
for drawing their attention to an issue in an earlier version of this paper, where we erroneously extended these arguments
to non-convex $\xi$. It remains a very interesting question to extend these results to this regime. The authors would like
to thank A. Montanari and G. Semerjian for drawing their attention to \cite{montanari2006rigorous}.
This research was conducted while G.BA. was supported by NSF DMS1209165, BSF 2014019 and A.J. was supported by NSF OISE-1604232.

\section{The landscape difficulty and spectral gap bounds for metric measure spaces \label{sec:The-difficulty-and-spectral-gap}}

We introduce here the notion of the landscape difficulty of a function and the maximal landscape
difficulty of a measure in our setting. The results of this section
do not depend of the rest of the of the paper and will be for deterministic
in a fixed class of metric measure spaces. 

Our goal is to produce upper bounds on the spectral gap of the infinitesimal
generator of a Markov process on some metric measure space. For us
this will be either a finite graph with a metric and measure or a
compact (weighted) Riemannian manifold. We wish for these estimates
to be intrinsic to the metric measure space. A classical approach
to prove such bounds is through isoperimetric methods: through Cheeger's
and Buser's inequalities \cite{Cheeger69,Bus82} in the manifold setting
and what is called alternately the Cheeger constant, Bottleneck ratio,
or conductance bound in the setting of graphs \cite{LawSok88,JerSin89,LPW,AlonMil85}.

Instead of working with sets, their complements, and surface areas,
we work with the volumes of a specific family of sets. We work with
the level sets of Lipschitz functions. Unlike, for example, the bottleneck
ratio, we want to be able to take these level sets to correspond to
disjoint energy windows, i.e. $\{f\in(E_{i}-\epsilon,E_{i}+\epsilon)\}$
for $\epsilon$ small enough and $E_{i}-E_{j}$ large enough. To this
end, we will use a modification of \cite[Proposition 21]{GJ16}. Though
in principle, this estimate is highly suboptimal asymptotically as
$\epsilon\to0$ (see \prettyref{rem:suboptimal}), provided the Lipschitz
constant of the function is itself smaller this is not a major issue. 

Let us now turn to the main results of this section. In the following
we say that $f\lesssim g$ if there is a universal constant $C$ such
that $f\leq Cg$ and that $f\lesssim_{a}g$ if $C$ depends at most
on $a$. If $A$ is a Borel subset of a metric space $(X,d)$ then
we define the $\epsilon$-dilation of this set by $A_{\epsilon}=\{x\in X:d(x,A)\leq\epsilon\}$.
Finally, $Lip_{K}$ denotes the space of $K$-Lipschitz functions. 

\subsection{The notion of landscape difficulty for a metric measure space}\label{sec:general-difficulty-def}
Let $(X,d,\nu)$ be a metric measure space. Let $f$ be a $K-$Lipschitz
function on $(X,d,\nu)$. For $E\in\R$ and $\epsilon>0$, let 
\begin{equation}
S(E,\epsilon;f)=\log\nu(f\in(E-\epsilon,E+\epsilon)).\label{eq:S-def}
\end{equation}
For every $C>0$, let 
\begin{equation}
\cR_{C}=\left\{ (E_{1},E_{2},E_{3})\in\R^{3}:E_{1}<E_{2}<E_{3},\:C<\frac{1}{4}\min\{E_{2}-E_{1},E_{3}-E_{2}\}\right\} .\label{eq:R-def}
\end{equation}
Define the function $\Phi:\cR_{\epsilon}\to\R$ by 
\begin{equation}
\Phi(E_{1},E_{2},E_{3},\epsilon;f)=S(E_{1},\epsilon;f)+S(E_{3},\epsilon;f)-S(E_{2},\epsilon;f).\label{eq:phi-def}
\end{equation}
Define the $\epsilon-$\emph{landscape difficulty} of $f$, or simply the $\epsilon$-\emph{difficulty} of $f$, by 
\begin{align}
\cD_{\epsilon}(f) & =\sup_{\substack{a,b,c\in\cR_{\epsilon}}
}\Phi(a,b,c,\epsilon;f).\label{eq:difficulty-def}
\end{align}
Finally, define the $(K,\epsilon)$-\emph{maximal landscape difficulty }of $\nu$, or simply
the $(K,\eps)$-maximal difficulty of $\nu$,
by 
\begin{equation}
\cD_{\epsilon}(\nu,K)=\sup_{f\in\Lip_{K}}\cD_{\epsilon}(f).\label{eq:max-diff-def}
\end{equation}
Here $\Lip_{K}$ is the space of $K$-Lipschitz functions. Note that $\cD_\eps(f)$ 
depends on $\nu$ as well, however, to distinguish this notation from $\cD_\eps(\nu,K)$,
we omit the dependence. 
We then
have the following definition. 
\begin{defn}
\label{def:FEB-and-difficult}Let $(X,d,\nu)$ be a metric measure
space. It is said that there is a \emph{free energy barrier corresponding
to f} if for some $\epsilon>0$, 
\[
\cD_{\epsilon}(f)>\log4.
\]
It is said that $\nu$ is $(K,\epsilon)$\textendash \emph{difficult}
if the maximal $(K,\epsilon)-$difficulty of $\nu$ satisfies 
\[
\cD_{\epsilon}(\nu,K)>\log4.
\]
\end{defn}

\subsection{Spectral Gap bounds for Finite Graphs}\label{sec:graphs-difficulty}

Let $G=(V,E)$ be a finite graph with a metric $d$ and measure $\nu$
on the vertex set $V$. We call the triple $(G,d,\nu)$ a \emph{metric
measure graph. }For two points $x,y\in V$, we say that $x\sim y$
if $x$ and $y$ are connected by an edge. Let $\Omega(x)=\{y:x\sim y\}$
denote the set of nearest neighbors of $x$. Let $D=\max_{x\sim y}d(x,y)$.
For a function $f:V\to\R$ on the vertex set of $G$, we define the
\emph{discrete gradient }by
\[
\nabla f=\left(f(x)-f(y)\right)_{y\in\Omega(x)}.
\]
For a set of vertices, $S$, let $\partial S$ be those vertices in
$S$ that have at least one edge leaving $S$. 

Let $Q$ be a transition matrix on $V$ with invariant measure $\nu$.
We say that $Q$ is \emph{nearest neighbor} if $Q(x,y)>0$ only if
$x\sim y$ or $x=y$. Let the Dirichlet form be given by 
\[
\cE(f,g)=\left((I-Q)f,g\right)_{\nu}.
\]
Recall that since $\nu$ is reversible, 
\[
\cE(f,f)=\frac{1}{2}\sum_{x}\sum_{y}\left(f(x)-f(y)\right)^{2}Q(x,y)\nu(x).
\]
Recall from Rayleigh's min-max principle
\cite{LPW} that, the spectral gap for $Q$, call it $\gamma_{Q}$,
satisfies
\begin{equation}
\gamma_{Q}=\min_{Var_{\nu}(f)\neq0}\frac{\cE(f,f)}{Var_{\nu}(f)}.\label{eq:min-max}
\end{equation}

\subsubsection{Upper bounds though the landscape difficulty: graphs}\label{sec:upper-bound-diff-graph}

In order to prove spectral gap upper bounds, we will wish to take
the following modification of the Conductance bound with respect to
sub-level sets of regular functions. Define the difficulty and maximal
difficulty for the metric measure space $(V,d,\nu)$ as in \prettyref{def:FEB-and-difficult}.
We then have the following theorem which allows us to bound the spectrum
of $I-Q$ using quantities that relate only to $(\nu,d).$
\begin{thm}
\label{thm:FE-barriers-graph} Let $(G,d,\nu)$ be a metric measure
graph. Let $Q$ be a transition matrix for that satisfies detailed
balance with respect to $\nu$ that is nearest-neighbor and has spectral
gap $\gamma_{Q}.$ Let $K>0$ and $\epsilon>2\cdot K\cdot D$. If
$\nu$ is $(K,\epsilon)$-difficult, then the spectral gap and the
maximum $(K,\epsilon)$-difficulty satisfy the relation
\begin{equation}
\gamma_{Q}\leq2\left(\frac{K\cdot D}{\epsilon}\right)^{2}\frac{e^{-\cD_{\epsilon}(\nu,K)}}{1-4e^{-\cD_{\epsilon}(\nu,K)}}.\label{eq:fe-barrier-inequality-graph}
\end{equation}
\end{thm}
\begin{rem}
This result uses the nearest-neighbor property rather weakly. In particular,
in the language of Bakry-Émery theory, if we let $\Gamma_{1}(f)(x)$
denote the Carré du champ 
\[
\Gamma_{1}(f)(s)=\frac{1}{2}\sum_{y}(f(x)-f(y))^{2}Q(x,y)
\]
then we use here simply that 
\[
\Gamma_{1}(f)(x)\leq\frac{1}{2}\norm{\nabla f}_{\infty}^{2}(x),
\]
This argument should extend to local dynamics, and even non-local
dynamics provided one assumes that $Q(x,y)$ decays sufficiently fast
as $d(x,y)$ increases, though we do not pursue this direction here.
\end{rem}
In order to prove this theorem, we start with the following estimate.
This is a modification and discretization of \cite[Proposition 21]{GJ16}.
\begin{lem}
\label{lem:discrete-conductance-bound} Let $Q$ be the transition
matrix on a metric measure graph $(G,d,\nu)$ that satisfies detailed
balance with respect to $\nu$, is nearest-neighbor, and has spectral
gap, $\gamma_{Q}$. Let $A\subset V$ and let $B=\partial A\cup\partial A_{\epsilon}^{c}\cup A_{\epsilon}\backslash A.$
Suppose that 
\[
\nu(A)\nu(A_{\epsilon}^{c})-4\nu(B)^{2}>0.
\]
Then $\gamma_{Q}$ satisfies
\[
\gamma_{Q}\leq\frac{D^{2}}{2\epsilon^{2}}\frac{\nu(B)}{\nu(A_{\epsilon}^{c})\nu\left(A\right)-4\nu(B)^{2}}.
\]
\end{lem}
\begin{proof}
We begin by the following simple observation. Since $Q$ is a transition
matrix, $\norm{Q(x,\cdot)}_{1}=1$ for every $x$. Thus by H\"older's
inequality and the fact that $Q$ is nearest-neighbor, we have that
for any function $f$ on $V$,
\begin{align}
\cE(f,f) & =\frac{1}{2}\sum_{x}\sum_{y}\left(f(x)-f(y)\right)^{2}Q(x,y)\nu(x)\leq\frac{1}{2}\sum_{x}\norm{\nabla f}_{\infty}^{2}(x)\nu(x).\label{eq:d-form-bound-conductance}
\end{align}
Let 
\[
\psi(x)=\begin{cases}
\nu(A) & x\in A_{\epsilon}^{c}\\
-\nu(A^{c}) & x\in A\\
-\nu(A^{c})+\min\{\frac{d(x,A)}{\epsilon},1\} & x\in A_{\epsilon}\backslash A
\end{cases}.
\]
We now bound the gradient of $\psi$. %
From the form of $\psi$, we obtain the gradient estimate
\[
\norm{\nabla\psi}_{\infty}^{2}(x)\leq\max_{y\in\Omega(x)}\frac{d(x,y)^{2}}{\epsilon^{2}}\indicator{x\in B}.
\]
Applying this to the estimate, \prettyref{eq:d-form-bound-conductance},
on the Dirichlet form gives the bound 
\begin{equation}
\cE(\psi,\psi)\leq\frac{D^{2}}{2\epsilon^{2}}\nu(B).\label{eq:d-form-bound-discrete}
\end{equation}
On the other hand, 
\[
\abs{\int\psi d\nu}\leq\abs{\int_{(A_{\epsilon}\backslash A)^c}\psi d\nu}+\abs{\int_{A_{\epsilon}\backslash A}\psi d\nu}\leq2\nu(A_{\epsilon}\backslash A),
\]
and 
\begin{align*}
\int\psi^{2}d\nu & \geq\int_{(A_{\epsilon}\backslash A)^c}\psi^{2}d\nu\\
 & =\nu(A)^{2}\left(\nu(A^{c})-\nu(A_{\epsilon}\backslash A)\right)+\nu(A^{c})^{2}\nu(A)\\
 & =\nu(A)\nu(A^{c})-\nu(A)^{2}\nu(A_{\epsilon}\backslash A)
\end{align*}
Thus 
\begin{align*}
Var_{\nu}(\psi) & \geq\nu(A)\nu(A^{c})-\nu(A)^{2}\nu(A_{\epsilon}\backslash A)-4\nu(A_{\epsilon}\backslash A)^{2}\\
 & =\nu(A)(\nu(A_{\epsilon}^{c})+\nu(A_{\epsilon}\backslash A))-\nu(A)^{2}\nu(A_{\epsilon}\backslash A)-4\nu(A_{\epsilon}\backslash A)^{2}\\
 & \geq\nu(A)\nu(A_{\epsilon}^{c})-4\nu(A_{\epsilon}\backslash A)^{2}
\end{align*}
Since $A_{\epsilon}\backslash A\subset B$, it follows that

\[
Var_{\nu}(\psi)\geq\nu(A)\nu(A_{\epsilon}^{c})-4\nu(B)^{2}.
\]

Thus the Rayleigh quotient satisfies 
\[
\frac{\cE(\psi,\psi)}{Var_{\nu}(\psi)}\leq\frac{D^{2}}{2\epsilon^{2}}\frac{\nu(B)}{\nu(A)\nu(A_{\epsilon}^{c})-4\nu(B)^{2}}.
\]
The result then follows by Rayleigh's min-max principle, \eqref{eq:min-max}.
\end{proof}
In the following, we will use a specific form of this estimate. 
\begin{cor}
\label{cor:conductance-lipschitz-function} Let $(G,d,\nu)$, $Q$,
and $\gamma_{Q}$ be as in \prettyref{lem:discrete-conductance-bound}.
Let $L\in\R$, let $f$ be a $K-$Lipschitz function on $V$. Suppose
that 
\begin{equation}
\nu(f\geq L)\nu(f\leq L-2K\delta\vee D)-4\nu(f\in(L-2K\delta\vee D,L+2K\delta\vee D))^{2}>0.\label{eq:denominator-discrete}
\end{equation}
Then 
\[
\gamma_{Q}\leq\frac{D^{2}}{2\delta^{2}}\frac{\nu\left(f\in(L-2K\delta\vee D,L+2K\delta\vee D)\right)}{\nu(f\geq L)\nu(f\leq L-2K\delta\vee D)-4\nu(f\in(L-2K\delta\vee D,L+2K\delta\vee D))^{2}}.
\]
\end{cor}
\begin{proof}
Let $A=\{f\geq L\}$. Observe that 
\begin{align*}
A_{\delta}  \subset\left\{ f\geq L-K\delta\right\} \qquad\qquad
\partial A  \subset\left\{ f\in[L,L+K\delta]\right\} .
\end{align*}
Suppose that $x\in\partial A_{\delta}^{c}$ . Then there is a $y\in A_{\delta}$
such that 
\[
d(y,x)\leq D.
\]
Then 
\[
f(x)\geq f(y)-K d(x,y)\geq L-K(\delta+D).
\]
Thus 
\[
\partial A_{\delta}^c\subset\{f\geq L-2K\delta\vee D\}.
\]
Thus if we let 
\[
\tilde{B}=\left\{ f\in(L-2K\delta\vee D,L+2K\delta\vee D)\right\} ,
\]
it follows that $B$ from \prettyref{lem:discrete-conductance-bound}
satisfies $B\subset\tilde{B}.$ Applying \prettyref{lem:discrete-conductance-bound},
then yields 
\[
\gamma_{Q}\leq\frac{D^{2}}{2\delta^{2}}\frac{\nu(\tilde{B})}{\nu(A)\nu(A_{\delta}^{c})-4\nu(\tilde{B})^{2}}.
\]
The result then follows by set containment.
\end{proof}
We can now prove \prettyref{thm:FE-barriers-graph}. 
\begin{proof}[\textbf{\emph{Proof of \prettyref{thm:FE-barriers-graph}}}]
\textbf{\emph{ }}Suppose that $\nu$ is $(K,\epsilon)$-difficult.
Then there is a $K$-Lipschitz $f$, an $\epsilon>2K\cdot D$, and
a triple $(E_{1},E_{2},E_{3})\in\cR_{\epsilon}$ such that 
\[
\Phi(E_{1},E_{2},E_{3};\epsilon)>\log4,
\]
and such that $S(E_{1},\epsilon;f),S(E_{2},\epsilon;f)>-\infty$.
Set $\delta=\frac{\epsilon}{2K}$, then $\delta>D$. Thus, if we let
$L=E_{2}$, then 
\begin{align*}
\nu(f\geq L)\nu(f\leq L-2K\delta\vee D) & -4\nu(f\in(L-2K\delta\vee D,L+2K\delta\vee D))^{2}\\
 & =\nu(f\geq E_{2})\nu(f\leq E_{2}-\epsilon)-4\nu(f\in(E_{2}-\epsilon,E_{2}+\epsilon))^{2}\\
 & \geq\nu(f\in(E_{3}-\epsilon,E_{3}+\epsilon))\nu(f\in(E_{1}-\epsilon,E_{1}+\epsilon))\left(1-4e^{-\Phi}\right).
\end{align*}
By assumption on $\Phi$ and $S$, and since $\delta>D$, we see that
\prettyref{eq:denominator-discrete} is positive. We may then apply
\prettyref{cor:conductance-lipschitz-function} to obtain 

\[
\gamma_{Q}\leq\frac{(2\cdot K\cdot D)^{2}}{2\epsilon^{2}}\frac{e^{-\Phi}}{1-4e^{-\Phi}}.
\]
Minimizing the right hand side in $E_{1},E_{2},E_{2},\epsilon,$ and
$f$ and using the fact that the function 
\begin{equation}
x\mapsto\frac{e^{-x}}{1-4e^{-x}}\label{eq:exp-function}
\end{equation}
is decreasing for $x\geq\log4$, yields the result. 
\end{proof}

\subsubsection{Lower bound by stability of Poincaré inequalities: graphs \label{subsec:Lower-bound-by-stability-discrete}}

This inequality allows us to prove spectral gap upper bounds. Correspondingly
it will be useful to obtain lower bounds. To this end, we remind the
reader of the following classical stability property of Poincaré inequalities due to Holley and Stroock \cite{HolStr87}.
See also \cite{GuiZeg03,SalCost97}.
\begin{prop}
\label{prop:poincare-stability-discrete} Let $(\cX,d,\mu)$ be a
finite metric measure space and let $d\nu=\frac{e^{-U(x)}}{Z}d\mu$,
as before. Suppose that $Q(x,y)$ is a transition matrix that satisfies
detailed balance with respect to $\nu$, and that $P(x,y)$ is a transition
matrix for a Markov chain that satisfies detailed balanced with respect
to $\mu$. Suppose furthermore that $AP(x,y)\leq Q(x,y)$. Then if
$P$ has spectral gap $\gamma_P$,  the spectral gap of 
$Q$ satisfies
\[
A e^{-2(\max U-\min U)}\gamma_P\leq \gamma_Q.
\]
\end{prop}

\subsection{Spectral Gap bounds for Compact Riemannian Manifolds}\label{sec:manifolds-difficulty}

Let $(M,g)$ be a smooth compact boundary-less Riemannian manifold
equipped with some measure 
\[
\nu=\frac{e^{-U}dvol}{Z},
\]
where $U\in C^\infty(M)$. For a function $f\in C^{\infty}(M)$ we let $Df$ denote the usual
gradient and we let $\Delta f$ denote the Laplace-Beltrami operator.
Let $L=\Delta-g(DU,D\cdot)$ be the corresponding Langevin operator,
and define the corresponding Dirichlet form
\[
\cE(f,h)=\int g(Df,Dh)d\nu.
\]
Note that $L$ is uniformly elliptic with smooth bounded coefficients
so its eigenfunctions are smooth by standard elliptic regularity \cite{GilTrud01,EvansPDE}.
Furthermore, $L$ is symmetric on $C^{\infty}(M)$ with respect to
$\nu$ so that in fact by this regularity result, one can show that
it is essentially self-adjoint there \cite{LaxFunctional02} and has
pure point spectrum $0=\lambda_{0}\leq\lambda_{1}\leq\ldots$ As a
result, the corresponding heat flow $P_{t}=e^{tL}$ is well defined.

By the Courant-Fischer min-max principle \cite{LaxFunctional02,Chav84},
recall that the first non-trivial eigenvalue of $L$ satisfies
\begin{align}
\lambda_{1} & =\min_{\substack{f\in C^{\infty}(M)\\
Var_{\nu}(f)\neq0
}
}\frac{\cE(f,f)}{Var_{\nu}(f)}\label{eq:Courant-min-max-principle}
\end{align}
That this is a minimum and not an infimum can be seen by elliptic
regularity \cite{EvansPDE,Chav84}. It will be useful to note that
one can relax this minimization problem to being over the space $H^{1}(\nu)\cap\{\norm{Df}_{L^{2}(\nu)}\neq0\}$. 

\subsubsection{Upper bound using the landscape difficulty: Riemannian manifolds}

As before, we seek to bound $\lambda_{1}$ using quantities that are
intrinsic to the metric measure space $(M,d_{g},\nu)$. Define the
difficulty and maximal difficulty as in \prettyref{def:FEB-and-difficult}
for $(M,d_{g},\nu)$. We then have the following theorem.
\begin{thm}
\label{thm:FE-barriers-manifold} Let $(M,g)$ be a smooth compact,
boundary-less Riemannian manifold, and let $\nu=\frac{e^{-U}dvol}{Z}$
for some smooth $U$. Let $L=\left(\Delta-g(DU,D\cdot)\right)$ with
first nontrivial eigenvalue $\lambda_{1}$. Let $K,\epsilon>0$. If
$\nu$ is $(K,\epsilon)-$difficult, then the spectral gap and the
maximum difficulty satisfy the relation
\begin{equation}
\lambda_{1}\leq\frac{K^{2}}{\epsilon^{2}}\frac{e^{-\cD_{\epsilon}(\nu,K)}}{1-4e^{-\cD_{\epsilon}(\nu,K)}}.\label{eq:fe-barrier-inequality-manifold}
\end{equation}
\end{thm}

\begin{rem}
\label{rem:gamma-1-grad-bound} Again, this result uses the form of
$L$ rather weakly. In particular, if we study a general reversible
dynamics with infinitesimal generator $L$, and let $\Gamma_{1}(f)(x)$
denote the corresponding Carré du champ, then the above result holds,
for example, if 
\[
\Gamma_{1}(f)(x)\leq Cg(Df,Df)^{2}(x),
\]
where the above inequality will have an additional factor of $C$.
\end{rem}
The proof of this is similar to the discrete setting. We begin, as
before, with the following which is a small modification of \cite[Proposition 21]{GJ16}.
\begin{lem}
\label{lem:(Conductance-Bound)} Let $(M,g)$ be smooth compact, boundary-less
Riemannian manifold, and let $A\subset M$ be Borel. Let $\nu=\frac{e^{-U}dvol}{Z}$
for some smooth $U$, let $L=\left(\Delta-g(DU,D\cdot)\right)$, and
let $\cE$ be its corresponding Dirichlet energy. Let $\lambda_{1}$
be first eigenvalue for $L$ restricted to the the orthogonal complement
of the constant functions. Let $B=A_{\epsilon}\backslash A$. Then
provided $\nu(A)\nu(A_{\epsilon}^{c})-4\nu(B)^{2}>0$ we have that
\[
\lambda_{1}\leq\frac{1}{\epsilon^{2}}\frac{\nu(B)}{\nu(A)\cdot\nu(A_{\epsilon}^{c})-4\nu(B)^{2}}.
\]
\end{lem}
\begin{rem}
\label{rem:suboptimal} As observed in \cite[Proposition 21]{GJ16},
this estimate is highly suboptimal in $\epsilon.$ Indeed as $\epsilon\to0$,
the numerator scales like $\epsilon$ so that the expression scales
like $\epsilon^{-1}$. See for example \cite{BakLed96,Led94}. In
our applications, however, this will be irrelevant. 
\end{rem}
\begin{proof}
Consider the test function 
\[
\psi(x)=\begin{cases}
\nu(A) & \text{on }(A_{\epsilon})^{c}\\
-\nu(A^{c}) & \text{on }A\\
-\nu(A^{c})+\min\left\{ \frac{d(x,A)}{\epsilon},1\right\}  & \text{on }B
\end{cases}
\]
Observe that since $d(x,A)$ is Lipschitz, $\psi\in H^{1}(dvol)$
and thus $H^{1}(d\nu)$. Observe furthermore that since $d(x,A)$
is Lipschitz, we have that 
\[
\norm{\nabla\psi}_{\infty}\leq\frac{1}{\epsilon}.
\]
Thus if we evaluate this on the Dirichlet form, we get 
\[
\cE(\psi,\psi)=\int_{M}g(D\psi,D\psi)d\nu=\int_{B}g(D\psi,D\psi)d\nu\leq\frac{1}{\epsilon^{2}}\nu(B).
\]
The variance lower bound is identical to that in \prettyref{lem:discrete-conductance-bound}.
Thus by the Courant-Fischer min-max principle \cite{LaxFunctional02},
\[
\lambda_{1}\leq\frac{\cE_{R}(\psi,\psi)}{Var_{\nu}(\psi)}\leq\frac{1}{\epsilon^{2}}\frac{\nu(B)}{\nu(A)\nu(A_{\epsilon}^{c})-4\nu(B)^{2}},
\]
as desired.
\end{proof}
We apply this for Lipschitz statistics. 
\begin{cor}
\label{cor:conductance-lipschitz-function-smooth} Let $(M,g)$ and
$\lambda_{1}$ be as in \prettyref{lem:(Conductance-Bound)}. Let
$L\in\R$, let $f$ be a $K-$Lipschitz function on $M$. Suppose
that
\begin{equation}
\nu(f\geq L)\nu(f\leq L-K\delta)-4\nu(f\in[L-K\delta,L))^{2}>0.\label{eq:denominator-manifold}
\end{equation}
 Then
\[
\lambda_{1}\leq\frac{1}{\delta^{2}}\frac{\nu(f\in[L-K\delta,L))}{\nu(f\geq L)\nu(f\leq L-K\delta)-4\nu(f\in[L-K\delta,L))^{2}}
\]
provided the denominator is positive. 
\end{cor}
\begin{proof}
Let 
\[
A=\left\{ f\geq L\right\} 
\]
Then
\begin{align*}
A_{\epsilon}  \subset\left\{ f\geq L-K\delta\right\} \qquad
A_{\epsilon}\backslash A  \subset\{f\in[L-K\delta,L)\} \qquad
A_{\epsilon}^{c}  \supset\{f\leq L-K\delta\}.
\end{align*}
The result then follows by \prettyref{lem:(Conductance-Bound)}. 
\end{proof}
We may then prove \prettyref{thm:FE-barriers-manifold}.
\begin{proof}[\textbf{\emph{Proof of \prettyref{thm:FE-barriers-manifold}}}]
 Suppose that $\nu$ is $(K,\epsilon)$-difficulty. Then there is
a $K$-Lipschitz $f$, and $\epsilon>0$, and a triple $(E_{1},E_{2},E_{3})\in\cR_{\epsilon}$
such that
\[
\Phi(E_{1},E_{2},E_{3},\epsilon;f)>\log4,
\]
and such that $S(E_{1},\epsilon;f),S(E_{3},\epsilon;f)>0$. Set $\delta=\frac{\epsilon}{K}$.
If we let $L=E_{2}$, then 
\[
\nu(f\geq L)\nu(f\leq L-K\delta)-4\nu(f\in[L-K\delta,L))^{2}\geq\nu(f\in(E_{3}-\epsilon,E_{3}+\epsilon))\nu(f\in(E_{1}-\epsilon,E_{1}+\epsilon))(1-4e^{-\Phi}).
\]
By the assumption on $\Phi$ and $S$, we see that \prettyref{eq:denominator-manifold}
is positive. We may then apply \prettyref{cor:conductance-lipschitz-function-smooth}
to obtain 
\[
\lambda_{1}\leq\frac{K^{2}}{\epsilon^{2}}\frac{e^{-\Phi}}{1-4e^{-\Phi}}.
\]
Then by \prettyref{cor:conductance-lipschitz-function-smooth}, if
we let $\delta=\frac{\epsilon}{K}$,
\begin{align*}
\lambda_{1} & \leq\frac{K^{2}}{\epsilon^{2}}\frac{\nu(f\in[E_{2}-\epsilon,E_{2}))}{\nu(f\geq E_{2})\nu(f\leq E_{2}-\epsilon)-4\nu(f\in[E_{2}-\epsilon,E_{2}))^{2}}.
\end{align*}
Minimizing the right hand side in $(E_{1},E_{2},E_{3}),\epsilon,$
and $f$ and using the fact that the function \prettyref{eq:exp-function}
is decreasing for $x\geq\log4$ yields the result
\end{proof}

\subsubsection{Lower bound by stability of Poincaré inequalities: manifolds}

We end this section again by noting the following classical result
regarding the stability of Poincaré inequalities due to Holley and Stroock \cite{HolStr87}. See also
\cite{GuiZeg03}.
 
\begin{prop}
\label{prop:poincare-stability-manifold} 
Suppose that $-\Delta$ has first non-trivial eigenvalue $\lambda_1 (M)$. 
Then $\lambda_1$, satisfies
\[
e^{-2(\max U-\min U)}\lambda_1(M)\leq \lambda_1
\]
\end{prop}

\section{Spectral Gap bounds in the presence of a free energy barrier}\label{sec:spectral-gap-and-FEB}
In this section, we show how a free energy barrier will imply spectral gap upper bounds for
dynamics of mean field spin glasses.
In particular, we aim to prove \prettyref{thm:hard-spin-FEB-thm} and \prettyref{thm:spherical-FEB-thm}.

Before turning to these proofs, we begin by observing the following elementary consequence of 
 concentration of measure. Recall that by Gaussian concentration \cite{Led01} for both spherical
 and Ising spin models, if
\[
Z_{N}(A)=\int\int_{R_{12}\in A}e^{-H(\sigma^{1})-H(\sigma^{2})}d\sigma^{1}d\sigma^{2},
\]
then there is a $K=K(\xi,h)$ such that 
\begin{equation}
\prob\left(\abs{\frac{1}{N}\log Z_{N}(A)-\E\frac{1}{N}\log Z_{N}(A)}>\epsilon\right)\leq Ke^{-N\epsilon/K},\label{eq:conc-gauss-1}
\end{equation}
for all $\epsilon>0$.  As a consequence, we have the following. 
 \begin{lem}
\label{lem:witness-lemma} Fix $\xi,h$. There is a constant $K=K(\xi,h)>0$ such
that the following holds for both Ising spin and spherical models.  Suppose that there is a relatively open
subset $A\subset[-1,1]$ such that
\[
\liminf\E\pi_{N}^{\tensor2}\left(R_{12}\in A\right)>0.
\]
Then 
\[
\prob\left(\frac{1}{N}\log\int\int_{R_{12}\in A}e^{-H(\sigma^{1})-H(\sigma^{2})}d\sigma^{\tensor2}-2F<-\epsilon\right)\leq Ke^{-N\epsilon/K}.
\]
\end{lem}
\begin{proof}
We prove this by contradiction. Let 
\[
\Delta_{N}=\frac{1}{N}\log\int\int_{R_{12}\in A}e^{-H(\sigma^{1})-H(\sigma^{2})}d\sigma^{\tensor2}-2F.
\]
Suppose, for contradiction, that 
\begin{equation}
\prob\left(\Delta_{N}<-\epsilon\right)\geq2Ke^{-N\epsilon/K},\label{eq:contradiction-ass}
\end{equation}
where $K$ is from \prettyref{eq:conc-gauss-1}. Then 
\begin{align*}
\prob\left(\E\Delta_{N}<-\epsilon/2\right) & \geq\prob\left(\abs{\Delta_{N}-\E\Delta_{N}}<\epsilon/2,\Delta_{N}<-\epsilon\right)
  \geq Ke^{-N\epsilon/K}
\end{align*}
by the inclusion-exclusion principle combined with \prettyref{eq:conc-gauss-1}
and \prettyref{eq:contradiction-ass}. Thus 
\[
\E\Delta_{N}\leq-\epsilon/2.
\]
Applying \prettyref{eq:conc-gauss-1} again, this implies that 
\[
\prob\left(\Delta_{N}\geq-\epsilon/4\right)\leq\prob\left(\abs{\Delta_{N}-\E\Delta_{N}}\geq\epsilon/4\right)\leq Ke^{-\frac{N\epsilon}{4K}}.
\]
Thus 
\[
\E\pi_{N}^{\tensor2}(R_{12}\in A)=\E e^{-N\Delta_{N}}\to0
\]
which is a contradiction.
\end{proof}
\subsection{Free energy barriers and the landscape difficulty of the overlap for Ising spin models}\label{sec:difficulty-FEB-ising}

To prove \prettyref{thm:hard-spin-FEB-thm}, let us first relate \FEB~to the difficulty. 
 In the following, we let $d_H$ denote
the unnormalized Hamming distance on $\Sigma_N^n$..

\begin{thm}
\label{thm:difficulty-thm-hypercube-new-version} For every $n\geq 1$ the following holds. Suppose that for
some $(K,\epsilon)$ with $\epsilon>2K$, $\pi_{N}^{\tensor n}$ is
$(K,\epsilon)$-difficult. Then 
\[
\frac{\lambda_{1}}{n}\leq2\left(\frac{K}{\epsilon}\right)^{2}\frac{e^{-\cD_{\epsilon}(\pi_{N}^{\tensor n},K)}}{1-4e^{-\cD_{\epsilon}(\pi_{N}^{\tensor n},K)}}.
\]
\end{thm}

\begin{proof}
Let us start with $n=1$. This follows immediately from \prettyref{thm:FE-barriers-graph}.
Indeed $Q$ is by assumption reversible with respect to $\pi_N$ and nearest
neighbor.  Furthermore we can think of $(\Sigma_N,d_H,\pi_N)$ 
as a metric measure graph in the obvious way.

Let us now take $n=2$. The case $n\geq 3$ is identical.
Recall the elementary observation that if we consider the \emph{replicated
transition matrix}, which is the transition matrix 
\begin{equation}
Q_{r}=\frac{1}{2}\left(Q\tensor Id+Id\tensor Q\right),\label{eq:replicated-dynamics-discrete-1}
\end{equation}
then $Q_r$ satisfies detailed balance with respect to $\pi_N^{\tensor 2}$ and  the spectral gap of $I-Q_{r}$ , 
\[
\Lambda_{1}=\min_{\Var_{\pi_{N}^{\tensor2}}(f)\neq0}\frac{((I-Q_{r})f,f)_{\pi^{\tensor2}}}{\Var_{\pi^{\tensor2}}(f)}
\]
 satisfies 
\begin{equation}
\Lambda_{1}=\frac{1}{2}\lambda_{1}.\label{eq:replicated-to-not-replicated-ising}
\end{equation}
This follows from the fact that the eigenbasis for $Q_{r}$ consists
of tensor products of the eigenbasis for $Q$. In the study of Markov
chains, $Q_{r}$ is often referred to as the transition matrix for
a product chain.  (See \cite{LPW} for this terminology. )

Observe that by \prettyref{eq:replicated-to-not-replicated-ising},
it suffices to prove that 
\begin{equation}\label{eq:big-lambda-goal}
\Lambda_{1}\leq2\left(\frac{K}{\epsilon}\right)^{2}\frac{e^{-\cD_{\epsilon}(\pi_{N}^{\tensor2},K)}}{1-4e^{-\cD_{\epsilon}(\pi_{N}^{\tensor2},K)}}.
\end{equation}
This follows immediately from \prettyref{thm:FE-barriers-graph}.
To see that we are in this setting, observe that we may view $(\Sigma_{N}\times\Sigma_{N},d_{H},\pi_{N}^{\tensor2})$
as a metric measure graph as follows. Let $G=(V,E)$ have vertex set
$V=\Sigma_{N}\times\Sigma_{N}$ and edge set 
\[
E=\left\{ (\bsig,\bsig')\in V\times V:d_{H}(\bsig,\bsig')=1\right\} .
\]
Thus $(G,d,\pi_{N}^{\tensor2})$ is a metric measure graph. Observe
that, $Q_{r}$ from \prettyref{eq:replicated-dynamics-discrete-1}
is a transition matrix that satisfies detailed balance with respect
to $\pi_{N}^{\tensor2}$ and is nearest neighbor since $Q$ satisfies
both of these properties. We are thus in the setting of \prettyref{thm:FE-barriers-graph}
for any $\epsilon>2K$, from which \prettyref{eq:big-lambda-goal} follows.
\end{proof}

With this in hand, we may then prove \prettyref{thm:hard-spin-FEB-thm}. 
\begin{proof}[\textbf{\emph{Proof of \prettyref{thm:hard-spin-FEB-thm}}}]
View the overlap map, $(\sigma^{1},\sigma^{2})\mapsto R_{12}$,
as a map on the metric measure graph $ (\Sigma_N^2,d_H,\pi_N^{\tensor 2})$ (we view 
this as a metric measure graph as in \prettyref{thm:difficulty-thm-hypercube-new-version}).  
Observe that $R_{12}$ is $N^{-1}$-Lipschitz.
Then the $\epsilon-$difficulty
of $R_{12}$ satisfies 
\[
\cD_\eps(\pi_N^{\tensor2};N^{-1})\geq \cD_\eps(R_{12})
\]
Suppose now that there is a free energy barrier of height $C>0$ corresponding to some $q_1,q_2,q_3$ and 
$\eps>0$.
By \prettyref{eq:exponential-rarity}, we have that for $N$ sufficiently large, 
\[
S(q_2,\eps;R_{12})=\frac{1}{N}\log \pi_N^{\tensor 2}(R_{12}\in(q_2-\eps,q_2+\eps)) < -C
\]
with probability $1-K_1e^{-N/K_1}$ for some $K_1>0$.  Similarly, by 
\prettyref{eq:separation} and \prettyref{lem:witness-lemma},
it follows that
\[
\frac{1}{N}\left(S(q_{1},\epsilon)+S(q_{3},\epsilon)\right)\geq-C/2
\]
with probability $1-K_2e^{-N/(2K_2)}$ for some $K_2>0$. Thus on the intersection of these
events, 
\[
\cD_{\epsilon}(R_{12})\geq\Phi(q_{1},q_{2},q_{3},\epsilon;R_{12})\geq N\frac{C}{2}.
\]
In particular, 
\[
\prob\left(\cD_{\epsilon}<\frac{C}{2}N\right)\leq K_3e^{-N/K_3},
\]
in this case for some $K_3>0$.  On the complement of this event, 
\[
\frac{1}{N}\log(\lambda_1)<-C/2
\]
by \prettyref{thm:difficulty-thm-hypercube-new-version}. The result then follows.
\end{proof}

\subsection{Free energy barriers and the landscape difficulty of the overlap for Spherical models}\label{sec:difficulty-FEB-spherical}
To prove \prettyref{thm:spherical-FEB-thm}, let us first relate \FEB~to the difficulty. 

We then have the following. 
\begin{thm}
\label{thm:difficulty-thm-spherical-new-version} For every $n\geq 1$ the following holds. Suppose that for
some $(K,\epsilon)$ with $K,\epsilon>0$, $\pi_{N}^{\tensor2}$ is
$(K,\epsilon)$-difficult. Then 
\[
\lambda_{1}\leq\left(\frac{K}{\epsilon}\right)^{2}\frac{e^{-\cD_{\epsilon}(\pi_{N}^{\tensor n},K)}}{1-4e^{-\cD_{\epsilon}(\pi_{N}^{\tensor n},K)}}.
\]
\end{thm}
\begin{proof}
In the case $n=1$ this immediately follows from \prettyref{thm:FE-barriers-manifold}.

Let us now take $n=2$. The case $n\geq 3$ is identical.
As in the Ising spin setting, it will be helpful to introduce the
replicated dynamics. The \emph{replicated dynamics} for spherical
models is the heat flow on the product space $\cS_{N}\times\cS_{N}$
induced by the generator
\[
\cL_{R}=\cL_{H}\tensor Id+Id\tensor\cL_{H},
\]
on $\cS_{N}^{2}$. Since $\cL_{H}$ is uniformly elliptic and essentially
self-adjoint, the same is true for $\cL_{R}$. In particular, its
spectrum is non-positive and pure point. Thus this heat flow is uniquely
defined. Heuristically, this corresponds to two particles, $(X_{t},Y_{t})$,
independently flowing with respect to the flow for $\cL_{H}$ . 

Recall that $\lambda_{1}$ is the first nontrivial eigenvalue of $\cL_{H}$
and, correspondingly, let $\Lambda_{1}$ denote the first non-trivial
eigenvalue of $\cL_{R}$. The starting point for our analysis is the
simple observation that 
\begin{equation}
\lambda_{1}=\Lambda_{1}.\label{eq:Efron-stein-1}
\end{equation}
To see this observe that the eigenfunctions of $\cL_{H}$ are a complete
basis for $L^{2}(\cS_{N},dvol)$, so their products are a complete
basis of $L^{2}(\cS_{N}\times\cS_{N},dvol^{\tensor2})$ by density
of tensor products. The result then follows by \prettyref{eq:Courant-min-max-principle}. 

 Again, by \prettyref{eq:Efron-stein-1}, observe that it suffices
to prove that 
\[
\Lambda_{1}\leq\left(\frac{K}{\epsilon}\right)^{2}\frac{e^{-\cD_{\epsilon}(\pi_{N}^{\tensor2},K)}}{1-4e^{-\cD_{\epsilon}(\pi_{N}^{\tensor2},K)}}.
\]
This is a consequence of \prettyref{thm:FE-barriers-manifold}. To
see that we are in this setting. Observe that $M=\cS_{N}\times\cS_{N}$
with the natural product metric is a compact boundary-less Riemannian
manifold and that 
\[
\nu=\pi_{N}^{\tensor2}=\frac{e^{-U}}{Z}dvol_{M}
\]
where $U(\sigma^{1},\sigma^{2})=H(\sigma^{1})+H(\sigma^{2}).$ Finally
observe that 
\[
\cL_{R}=-\Delta+g(DU,D\cdot).
\]
Thus we are in the setting of \prettyref{thm:FE-barriers-manifold}
for any $K,\epsilon>0$.
\end{proof}
Finally we note the following.
\begin{proof}[\textbf{\emph{Proof of \prettyref{thm:spherical-FEB-thm}}}]
 This result follows from \prettyref{thm:difficulty-thm-spherical-new-version}  after observing that the
overlap map is $N^{-1/2}$-Lipschitz. The proof is identical to \prettyref{thm:hard-spin-FEB-thm}
so it is omitted. 
\end{proof}

\subsection{Spectral gap lower bounds}\label{subsec:Spectral-Gap-lower}

We end this section by briefly mentioning the spectral gap lower bounds for Ising spin and spherical models. 

We first briefly turn to the proof of {{\prettyref{thm:coercive-mixing}}}.
Recall from \cite{ABA13}, that by an application of Borell's and
Slepian's inequalities,
\begin{equation}
\prob(-CN\leq\min_{\sigma\in\cS_{N}}H(\sigma)\leq\max_{\sigma\in\cS_{N}}H(\sigma)\leq CN)\geq1-\frac{1}{c}e^{-cN}\label{eq:max-bound}
\end{equation}
for some $C=C(\xi,h)>0\text{ and }c=c(\xi,h)>0$. The same bound
then holds for the Ising spin setting since $\Sigma_N\subset\cS_N$. 

In order to obtain an exponential lower bound in the Ising spin setting,
recall that we needed coercivity. Recall that the spectral gap of the
simple random walk \cite{DiaSal96} is 
\[
\lambda_{SRW}=\frac{2}{N}.
\]

\begin{proof}[\textbf{\emph{Proof of \prettyref{thm:coercive-mixing}}}]
 Observe that by \prettyref{prop:poincare-stability-discrete}, if
$Q$ is $P_{SRW}-$coercive with constant $A_N$, then 
\[
\lambda_{1}(Q)\geq A_N e^{-2(\max H-\min H)}\lambda_{SRW}.
\]
Taking logs and using \prettyref{eq:max-bound}, we see that 
\[
\lambda_{1}(Q)\geq \frac{1}{N}\log{A_N}  -2 C
\]
for $N$ sufficiently large
\end{proof}
\begin{proof}[\textbf{\emph{Proof of \prettyref{cor:met-coercive}}}]
It suffices to show that 
\[
\frac{1}{N}\log(A_N) = -(\max H_N - \min H_N)
\]
To see this, simply observe that when $\sigma^1\neq\sigma^2$,
\[
Q(\sigma^1,\sigma^2)=P_{SRW}(\sigma^1,\sigma^2)(1\wedge e^{H(\sigma^2)-H(\sigma^1)}),
\]
and  $P_{SRW}(\sigma^1,\sigma^1)=0$. 
\end{proof}

In the setting of spherical models we have a similar result. Recall that the first non-trivial
eigenvalue for the Laplacian on $\cS_N$ satisfies \cite{Chav84} 
\[
\lambda_1(\cS_N) = 1-\frac{1}{N}.
\]
\begin{proof}[\textbf{\emph{Proof of \prettyref{thm:sg-lower-bound-spherical}}}]
To obtain the spectral gap lower bound from \prettyref{eq:sg-lower-bound-spherical}
in the spherical spin setting, observe that \prettyref{eq:max-bound}
still applies. The result then follows by
\prettyref{prop:poincare-stability-manifold}. 
\end{proof}
 
\section{Free energy barriers in Ising spin models\label{sec:Exponential-rarity-ising-spin}}
In this section, we aim to prove \prettyref{thm:hard-spin-main-thm}.
The main difficulty is showing that 
 that certain regions of overlap values are exponentially rare as in \prettyref{eq:exponential-rarity}.
This is the content of the following theorem, which is the goal of this section. 
\begin{thm}
\label{thm:exp-rare} Suppose that for some $q_{*}$ in the support
of $\pmeas$, $\Lambda(q_*,\pmeas)>0$. Then there is an $\epsilon_{0}$
such that for every $q$ in the punctured neighborhood $(q_{*}-\epsilon_0,q_{*}+\epsilon_0)\cap(0,1)\backslash\{q_{*}\}$,
there is an $\epsilon(q)$ and a $c(q)>0$ such that 
\[
\limsup_{N\to\infty}\frac{1}{N}\log\prob\left(\frac{1}{N}\log\pi_{N}^{\tensor2}(R_{12}\in(q-\epsilon,q+\epsilon))>-c\right)< 0
\]
If, furthermore, $q_{*}$ is in the support of $\zeta$ from \prettyref{eq:overlap-distn},
then it must be isolated.
\end{thm}

To prove this estimate, we control constrained free energies:
\begin{equation}
F_{2,N}(A)=\frac{1}{N}\log\int\int_{R_{12}\in A}e^{-H(\sigma^{1})-H(\sigma^{2})}d\sigma^{\tensor 2}
\end{equation}
where $A$ is some Borel set. More precisely, taking $A=(q-\eps,q+\eps)$, we will show that
\begin{equation}
F_{2,N}((q-\eps,q+\eps))=\frac{1}{N}\log\int\int_{\abs{R_{12}-q}<\eps}e^{-H(\sigma^{1})-H(\sigma^{2})}d\sigma^{\tensor2}
\end{equation}
\begin{equation}
F_{N}=\frac{1}{N}\log\int e^{-H(\sigma)}d\sigma\label{eq:fe-def}
\end{equation}
satisfy 
\[
F_{2,N}((q-\eps,q+\eps))-2F_{N}<-c
\]
 with high probability. This will follow by application of the 2D Guerra-Talagrand bounds. The key ideas in this proof can already be seen in \cite{TalPF} and \cite{TalBK11Vol2}. 
For completeness, we present here an alternative, stochastic analysis and PDE based approach following Bovier--Klimovsky  \cite{bovier2009aizenman} and Chen \cite{Chen15}.

\textbf{Notation:} Here and in the following, for a probability measure
$\nu$ we make the abuse of notation $\nu(t)=\nu([0,t])$. All matrix norms will be Frobenius/Hilbert-Schmidt norms.

\subsection{2D Guerra-Talagrand Bounds.}\label{sec:2d-gt-ising}

Let $\PSD_{d}$ be the space of $d\times d$ positive semidefinite
matrices. Fix $q\in[-1,1].$ Let $Q_{t}:[0,1]\to\PSD_{2}$, be
a continuous, weakly differentiable, non-decreasing
path in $\PSD_{2}$ with boundary conditions
\begin{align*}
Q_{0} & =0\\
Q_{1} & =\left(\begin{array}{cc}
1 & q\\
q & 1
\end{array}\right).
\end{align*}
(Here, by weakly differentiable we mean in the sense that its derivative
in $t$ is $W^{1,1}(\R;(\PSD_{2},\norm{\cdot}))$.) Let the space of such paths be denoted by $\mathcal{Q}_{q}$.
Let the space of such paths with arbitrary final data $Q_1$ be denoted by 
$\mathcal{Q}$.

 Let
$\nu\in\Pr([0,1])$. Finally, let 
\[
A=\frac{d}{dt}\left(\xi'(Q_{t})\right)=\xi''(Q_{t})\odot\dot{Q}_{t}
\]
where $\odot$ denotes the Hadamard product and function evaluations
are to be understood component wise. Since $Q$ was assumed to be
non-decreasing, $\dot{Q}$ is positive semidefinite. We observe here
the following lemma.
\begin{lem}
For any $Q\in\cQ$, $\xi''(Q_{t})$ and $A_{t}$ are positive semidefinite for each $t$. 
\end{lem}
This follows by Schur's product theorem after observing that $\xi''(Q_{t})$
can be viewed as a power series in $Q$ in the Hadamard product sense
and $\dot{Q}$ is positive semidefinite.

Let us begin by supposing that $A_{t}$ is strictly positive definite
for all $t$. We may then consider the weak solution, $u$, of 
\begin{equation}
\begin{cases}
\partial_{t}u+\frac{1}{2}\left(\left(A,D^{2}u\right)+\nu(t)\left(Du,ADu\right)\right)=0\\
u(1,x)=f_{\lambda}(x)
\end{cases}\label{eq:multidim-ppde}
\end{equation}
where 
\[
f_{\lambda}(x)=\log\left(\frac{1}{4}\sum_{\epsilon_{1},\epsilon_{2}\in\{\pm1\}}\exp\left(\epsilon_{1}x_1+\epsilon_{2}x_2+\lambda\epsilon_{1}\epsilon_{2}\right)\right).
\]
For the existence, uniqueness, and basic regularity of $u$, see \prettyref{app:Analytical-Properties-of}.
For any $Q\in\cQ,\nu\in\Pr([0,1])$ and $\lambda\in\R$, define the quantities
\begin{equation}
L(\nu,Q)=\frac{1}{2}\sum_{ij}\int\nu(t)\xi''(q_{ij}(t))q_{ij}(t)\dot{q}_{ij}(t)dt,\label{eq:L-def}
\end{equation}
and 
\begin{equation}
P(\nu,Q,\lambda)=u_{\nu}(0,h)-\lambda q-\frac{1}{2}\sum_{ij}\int\nu(t)\xi''(q_{ij}(t))q_{ij}(t)\dot{q}_{ij}(t)dt.\label{eq:GT-upperbound}
\end{equation}
Finally, let $\cR_{N}$ denote the set of allowed overlaps, 
\[
\cR_{N}=\left\{ q\in[-1,1]:\exists\sigma^{1},\sigma^{2}\in\Sigma_{N}:R_{12}=q\right\} .
\]
 then have Talagrand's 2D Guerra-Talagrand bound \cite{TalPF,TalBK11Vol2}.
\begin{thm}
We have the following:
\begin{enumerate}
\item If $\xi$ is convex on $[-1,1]$, then for every $q\in\cR_{N},Q_t \in\cQ_q$ positive definite, $\nu\in \Pr([0,1])$ and 
$\lambda\in\R$, 
\begin{equation}
\E F_{2,N}(\{q\})\leq P(\nu,Q,\lambda) \label{eq:GT-inequality-ising}
\end{equation}
\item In particular, if $\xi$ is convex on $[-1,1]$, then for
every $q\in[-1,1]$, 
\begin{equation}
\lim_{\epsilon\to0}\limsup_{N\to\infty}\E F_{2,N}((q-\epsilon,q+\epsilon)\cap[0,1]))\leq P(\nu,Q,\lambda).\label{eq:GT-inequality-ising-1}
\end{equation}
\end{enumerate}
\end{thm}

Let us now turn to the setting in which we will apply this class of
estimates. In our applications, we will be interested in cases where
$Q_{t}$ is allowed to be positive semi-definite. We will focus on
a specific form. In particular, take $q\geq0$ and define $Q_t(q)\in\cQ_q$ by
\begin{equation}
Q_{t}(q)=\begin{cases}
\left(\begin{array}{cc}
t & t\\
t & t
\end{array}\right) & t\leq q\\
\left(\begin{array}{cc}
t & q\\
q & t
\end{array}\right) & t\geq q
\end{cases}.\label{eq:q-degnerate-def}
\end{equation}
In this case 
\begin{equation}
A(t)=\begin{cases}
\xi''(t)\indicator{} & t\leq q\\
\xi''(t)Id & t\geq q
\end{cases},\label{eq:A-degenerate-def}
\end{equation}
where $\indicator{}$ is the matrix of all $1$'s.
Define 
\begin{equation}
P(\nu,Q(q),\lambda)=v(0,h)-\lambda q-L(\nu,Q),\label{eq:degenerate-multid-pfunc}
\end{equation}
where $v(t,x)=u(t,x,x)$ for $t\geq q$ and $v(t,x)$ is the unique
weak solution of 
\begin{equation}
\begin{cases}
\partial_{t}v+\frac{\xi''}{2}\left(\Delta v+\nu(t) (\partial_x v)^{2}\right)=\xi''(t)\left\{ \partial_{x_{1}}\partial_{x_{2}}u(t,x,x)+\nu(t)\partial_{x_{1}}u(t,x,x)\partial_{x_{2}}u(t,x,x)\right\} \indicator{t\geq q} & (t,x)\in[0,1]\times\R^{2}\\
v(q,x)=u(q,x,x) & .
\end{cases}\label{eq:v-def-pde}
\end{equation}
For the notion of weak solution in this setting and the existence
and uniqueness see \prettyref{app:Analytical-Properties-of}. We then
have the following, which is proved by  a standard extension argument. See, e.g.,
\cite{Chen15}.
\begin{cor}
\label{cor:degenerate-GT-bound}For every $q\in[0,1]$, $Q$ as in
\prettyref{eq:q-degnerate-def}, $\nu\in\Pr([0,1])$, and $\lambda\in\R$,
we have 
\begin{equation}
\lim_{\epsilon\to0}\limsup_{N\to\infty}\E F_{2,N}((q-\epsilon,q+\epsilon)\cap[0,1])\leq P(\nu,Q,\lambda).\label{eq:GT-inequality-ising-degenerate}
\end{equation}
\end{cor}

\subsection{Bounding the 2D Guerra-Talagrand functional under the assumption
of a positive replicon eigenvalue}\label{sec:nrev-gtbound-ising}

Consider the probability measure $\nu$, defined by the
cumulative distribution function
\begin{equation}
\nu(t)=\begin{cases}
\frac{\mu(t)}{2} & t\leq q\\
\mu(t) & t\geq q
\end{cases},\label{eq:nu-def}
\end{equation}
where $q\in[0,1]$ and $\mu$ is the Parisi measure. Let $Q_{t}(q)$ be as in \prettyref{eq:q-degnerate-def},
and let $A_{t}$ be as in \prettyref{eq:A-degenerate-def}. \prettyref{cor:degenerate-GT-bound} applies in this setting. 
Observe that in this setting the functional
\prettyref{eq:degenerate-multid-pfunc}, is a function of $q$ and
$\lambda$ alone, so we denote it by 
\begin{equation}
\cP(\lambda,q)=P(\nu,Q,\lambda).\label{eq:cP-def}
\end{equation}
 We aim to prove the following theorem. Recall that $\mu$ is the Parisi measure from \prettyref{def:parisi-measure}.
\begin{thm}
\label{thm:hard-spin-barrier} Let $q_{*}\in\supp(\mu)$ be such that
$\Lambda_{R}(q_{*},\mu)>0$. Then there is an  $\eps_0$ 
such that for all $q\in(q_{*}-\epsilon_{0},q_{*}+\epsilon_{0})\cap(0,1)$
with $q\neq q_{*}$, there is a $\lambda_*(q)$ satisfying
\[
\cP(\lambda_{*},q)<2P_{I}(\mu).
\]
\end{thm}

We begin the proof of \prettyref{thm:hard-spin-barrier} with the
following elementary observations. Observe that by \prettyref{eq:q-degnerate-def}
and \prettyref{eq:nu-def}, $L$ from \prettyref{eq:L-def} is constant
in $(\lambda,q)$ and satisfies 
\[
L=\int\xi''(t)t\mu(t)dt.
\]
Observe furthermore, that at $\lambda=0$, $u_{\nu}$ from \prettyref{eq:multidim-ppde}
with parameters \prettyref{eq:nu-def} and \prettyref{eq:A-degenerate-def},
factorizes for $t\geq q$ as
\begin{equation}
u_{\nu}(t,x,y)=\phi_{\mu}(t,x)+\phi_{\mu}(t,y),\label{eq:q-factorize}
\end{equation}
where $\phi_{\mu}$ is the solution of the Parisi initial value problem,
\prettyref{eq:PPDE-IVP-ising}, corresponding to $\mu$. By a scaling
argument applied to the Parisi PDE, since $\phi_{\mu}$ satisfies
\prettyref{eq:PPDE-IVP-ising} and $\nu$ satisfies \prettyref{eq:nu-def},
$2\phi_{\mu}$ is the solution, $v,$ of \prettyref{eq:v-def-pde},
\begin{equation}
v=2\phi_{\mu}\label{eq:v-lambda-0}
\end{equation}
 for all $(t,x)\in[0,1]\times\R.$ Thus 
\begin{equation}
\cP(0,q)=2\left(\phi_{\mu}(0,h)-\frac{1}{2}L\right)=2P_{I}(\mu),\label{eq:pfunc-at-lambda-0}
\end{equation}
for all $q$. 

Let us now explain, formally, the argument behind \prettyref{thm:hard-spin-barrier}.
By \prettyref{eq:pfunc-at-lambda-0}, $\cP$ is constant on the $\lambda=0$
axis. As we will soon see, the point $(\lambda,q)=(0,q')$ is
a critical point for $\cP$ for any $q'$ in $\supp(\mu).$ Evidently,
$\partial_{q}^{2}\cP(0,q')=0$ for such $q'$. Thus, formally, the
Hessian of $\cP$ is of the form 
\begin{equation}
Hess(\cP)=\left(\begin{array}{cc}
a & b\\
b & 0
\end{array}\right)\label{eq:hess-p}
\end{equation}
for some $a,b\in\R$. Note that this has a negative eigenpair 
\begin{equation}
\begin{cases}
\lambda & =\frac{1}{2}\left(a-\sqrt{a^{2}+4b^{2}}\right)\\
v & =\left(\frac{a-\sqrt{a^{2}+4b^{2}}}{b},1\right)
\end{cases}\label{eq:eigenpair}
\end{equation}
provided $b\neq0$. What we will find is that, 
\[
b=\partial_{q}\partial_{\lambda}P=-\Lambda_{R}(q)>0,
\]
which will yield the result. 

Rigorously, it is cumbersome to check that $\cP$ is jointly $C^{2}$.
To avoid this issue we recall the following basic result of calculus
which is a minor modification of the second derivative test.
\begin{lem}
\label{lem:calculus-lemma}Let $f(x,y)$ be a continuous function
of two variables such that:
\begin{enumerate}
\item it has a partial derivative in $x$ at $(x_{0},y_{0})$ that vanishes at that
point,
\item it has a  locally bounded, continuous second partial derivative in $x$ for
all $(x,y)$,
\item it has a nonzero mixed partial derivative at $(x_{0},y_{0})$, $\partial_{y}\partial_{x}f$,
\item $f(x_{0},y)$ is constant in $y$.
\end{enumerate}
Then there is an $r$ such that for all $y\in B_{r}(y_{0})\setminus\{y_0\}$, there
is an $x_{*}(y)$ with
\[
f(x_{*},y)<f(x_{0},y_{0}).
\]
Furthermore, the same holds if the mixed partial in $y$ is only a
right (left) derivative except with $0<y-y_{0}<r$  (resp. $0>y-y_0>-r$).%
\end{lem}
With this in mind, let us begin the proof. 

\subsubsection{Derivatives of the Multidimensional Parisi PDE in the Lagrange multiplier}

We start with the following result, regarding the differentiability
of $v$ in $\lambda$. Such results are standard in the spin glass
literature. See, e.g., \cite{chen2015fluctuations,TalBK11Vol2}. Let
$\mathbf{\tilde{X}}_{s}=(\tilde{X}_{s}^{1},\tilde{X}_{s}^{2})$ be
the solution of
\begin{equation}
d\tilde{\mathbf{X}}_{t}=\nu ADu_{\nu}(t,\mathbf{\tilde{\mathbf{X}}_{t}})dt+\sqrt{A}d\mathbf{W}_{t},\label{eq:tilde-X-multid}
\end{equation}
for $t\geq q$, where $\mathbf{W}_{t}$ is standard Brownian motion
in $\R^{2}$ and let $\tilde{X}_{t}$ be the solution of 
\begin{equation}
d\tilde{X}_{t}=\xi''\nu(t)\partial_x v(t,\tilde{X}_{t})dt+\sqrt{\xi''(t)}dW_{t},\label{eq:tilde-X}
\end{equation}
where $W_{t}$ is standard Brownian motion for $t\leq q$. Note that
since $Du$ and $\partial_x v$ are bounded measurable in time and uniformly
Lipschitz in space (in fact they are smooth and bounded in space) by \prettyref{lem:regularity-u}
and \prettyref{lem:regularity-v}, these solutions exist in the It\^o
sense. (The regularity of $u,v,$ and the Parisi PDE are discussed
in \prettyref{app:Analytical-Properties-of}) 
\begin{lem}
\label{lem:derivatives} We have the following.
\begin{enumerate}
\item The solution $u$ of \prettyref{eq:multidim-ppde} with parameters
given by \prettyref{eq:A-degenerate-def} and \prettyref{eq:nu-def}
is twice differentiable in $\lambda$ for $(t,x,y)\in[q,1]\times\R^{2}$
for each $q$. Furthermore, $\partial_{\lambda}u$ satisfies 
\[
\partial_{\lambda}u(t,x,y)=\E\left(\partial_{\lambda}f(1,\tilde{\mathbf{X}}_{1})\vert\tilde{\mathbf{X}}_{t}=(x,y)\right).
\]
\item The solution $v$ of \prettyref{eq:v-def-pde} with parameter given
by \prettyref{eq:nu-def} is twice differentiable in $\lambda$ for
each $(t,x)$ with $t<q$ and each $q$ . Furthermore, $\partial_{\lambda}v$
satisfies
\begin{equation}
\partial_{\lambda}v(0,h)=\E\left(\partial_{\lambda}u(q,\tilde{X}_{q},\tilde{X}_{q})\vert\tilde{X}_{0}=h\right).\label{eq:diff-v-lambda}
\end{equation}
Finally, the first and second derivatives are continuous in $(\lambda,q)$
and uniformly bounded in $(t,x)$ and $(\lambda,q)$. 
\end{enumerate}
\end{lem}
The proof of this result is a standard differentiable dependence argument.
Since it is technical, we defer it to \prettyref{subsec:Proof-of-deriv-lem}. 

Let us now compute the derivatives in which we are interested. In
the following, we let $X_{t}$ denote the solution to the local fields
process \prettyref{eq:local-fields}.
\begin{lem}
\label{lem:derivatives-evaluated} For every $q\geq0$, at $\lambda=0$,
\begin{align}
\partial_{\lambda}u(q,x,x) & =\left(\partial_x \phi_{\mu}\right)^{2}(q,x)\label{eq:u-deriv-lambda-0}\\
\partial_{\lambda}v(0,x) & =\E_{x}\left(\partial_x \phi_{\mu}\right)^{2}(q,X_{q})\label{eq:v-deriv-lambda-0}
\end{align}
where $X_{t}$ is the local field process \prettyref{eq:local-fields}
with initial data $X_{0}=x$. Furthermore, at $\lambda=0$, $\partial_{\lambda}v(0,h)$
has a partial derivative in $q$ and the derivative satisfies 
\begin{equation}
\partial_{q}\partial_{\lambda}v(0,h)=\xi''(q)\E\left(\Delta \phi_{\mu}\right)^{2}(q,X_{q}),\label{eq:mixed-partial}
\end{equation}
where if $q=0,$ this is a right-partial derivative, and if $q=1$ this is a left-partial derivative.
\end{lem}
\begin{proof}
Observe that at $\lambda=0$, 
\[
\partial_{\lambda}f_{\lambda}(x,y)=\tanh(x)\cdot\tanh(x).
\]
Recall from \prettyref{eq:PPDE-IVP-ising}, that $\partial_x \phi_{\mu}(1,x)=\tanh(x)$.
Thus 
\[
\partial_{\lambda}u(q,x,x)=\E\left(\partial_x \phi_{\mu}(1,\tilde{X}_{1}^{1})\cdot\partial_x \phi_{\mu}(1,\tilde{X}_{1}^{2})\vert(\tilde{X}_{t}^{1},\tilde{X}_{t}^{2})=(x,x)\right),
\]
by \prettyref{lem:derivatives}. Since $t\geq q,$ $u$ satisfies
\prettyref{eq:q-factorize} and $A(t)=\xi''(t)Id$. Thus $\tilde{\mathbf{X}}$
from \prettyref{eq:tilde-X-multid} is two independent copies of the
local field process, \prettyref{eq:local-fields}, corresponding to
$\mu$, which we denote by $(X_{t}^{1},X_{t}^{2})$. Thus 
\[
\partial_{\lambda}u(q,x,x)=\E\left(\partial_x \phi_{\mu}(1,X_{1}^{1})\cdot\partial_x \phi_{\mu}(1,X_{1}^{2})\vert(X_{1}^{1},X_{1}^{2})=(x,x)\right).
\]

Observe that $\partial_x \phi_{\mu}$ weakly solves 
\begin{align*}
\left(\partial_{t}+\frac{1}{2}\cL\right)\partial_x \phi_{\mu} & =0,
\end{align*}
where $\cL$ is the infinitesimal generator of $X_{t}$, 
\begin{equation}
\cL=\frac{\xi''(s)}{2}\left(\Delta+2\mu(s)\partial_x \phi_{\mu}(s,x)\partial_{x}\right).\label{eq:AC-process-inf-gen}
\end{equation}
 Thus, $\partial_x \phi_{\mu}(s,X_{s})$ is a martingale. By the martingale
property and independence, we then obtain 
\[
\partial_{\lambda}u(q,x,x)=\partial_x \phi_{\mu}(q,x)^{2}.
\]
This is the first equality. 

We now turn to the second. With \prettyref{eq:u-deriv-lambda-0}, we see that \prettyref{eq:diff-v-lambda}
satisfies 
\[
\partial_{\lambda}v(0,h)=\E\left(\partial_x \phi_{\mu}\right)^{2}(q,\tilde{X}_{q}),
\]
where $\tilde{X}_{t}$ solves \prettyref{eq:tilde-X}. Differentiating
\prettyref{eq:v-lambda-0} in space and applying \prettyref{eq:nu-def},
we see that $\tilde{X}_{t}=X_{t}$. This yields the second result. 

By an application of It\^{o}'s lemma, we have that
\begin{equation}
\frac{d}{dt}\E(\partial_{x}\phi_{\mu})^{2}(t,X_{t})=\xi''(t)\E\left(\Delta \phi_{\mu}\right)^{2}(t,X_{t}),\label{eq:ito-isometry}
\end{equation}
for every $t\geq0$, where if $t=0$ this is a right derivative and if $t=1$ this is a left derivative. Since
$\partial_{\lambda}v(0,h)$ satisfies \prettyref{eq:v-deriv-lambda-0}
at $\lambda=0$ for every $q\geq0$, we see that \prettyref{eq:ito-isometry}
implies that the partial derivative in $q$ at $\lambda=0$ satisfies \prettyref{eq:mixed-partial}\end{proof}

\subsubsection{Proof of \prettyref{thm:hard-spin-barrier} }

With these results in hand we may then prove \prettyref{thm:hard-spin-barrier}.
\begin{proof}[\textbf{\emph{Proof of \prettyref{thm:hard-spin-barrier}}}]
 We note the following. Firstly, by \prettyref{lem:derivatives-evaluated}
and \prettyref{eq:pfunc-at-lambda-0}, $\cP(\lambda,q)$ has a partial
derivative in $\lambda$ at $0$ that satisfies 
\[
\frac{\partial}{\partial\lambda}\cP(0,q)=\E\left(\partial_x \phi_{\mu}\right)^{2}(q,X_{q})-q.
\]
 Since $\mu$ is the Parisi measure by assumption, for every $q'$
in the support of $\mu$, we have the fixed point relation 
\[
\E\left(\partial_x \phi_{\mu}\right)^{2}(q',X_{q'})=q',
\]
by \prettyref{eq:optimality-conditions}. Thus for $q'$ in the support
of $\mu$, 
\[
\frac{\partial}{\partial\lambda}\cP(0,q')=0.
\]
Fix $\lambda=0$, and now differentiate again in $q$. We then obtain 
\[
\frac{\partial}{\partial q}\frac{\partial}{\partial\lambda}\cP(0,q)=\xi''(q)\E\left(\Delta \phi_{\mu}\right)^{2}(q,X_{q})-1=-\Lambda_{R}(q)
\]
by \prettyref{lem:derivatives-evaluated}, where if $q=0$, this is
understood to be a right derivative and if $q=1$ this is understood to be a right derivative. This is negative at $q=q_*$ by assumption.

By \prettyref{lem:derivatives}, $\partial_{\lambda}^{2}\cP$ is uniformly
bounded and continous in a neighborhood of $(0,q')$ for any $q'$ in the support
of $\mu$. Recall that $\cP(0,q)$ is constant in $q$. 
The result then follows
by \prettyref{lem:calculus-lemma}, applied to $\cP$ at the point $(0,q_*)$.
\end{proof}

\subsection{Proof of \prettyref{thm:exp-rare}}\label{sec:exp-rare-ising-pf}

As a consequence of the previous section, we may also prove \prettyref{thm:exp-rare}.

\begin{proof}[\textbf{\emph{Proof of \prettyref{thm:exp-rare}}}]
Let $\eps_0$ and 
\[
E= (q_*-\eps_0,q_*+\eps_0)\cap(0,1)\setminus\{q_*\}
\]
be as in \prettyref{thm:hard-spin-barrier}. Fix $q\in E$ and let $\lambda_*(q)$ be as in \prettyref{thm:hard-spin-barrier}. 
By \prettyref{thm:hard-spin-barrier},
\prettyref{eq:parisi-formula-ising-minimizing}, and \prettyref{cor:degenerate-GT-bound},
it follows that 
\[
\E F_{2,N}((q-\eps,q+\eps))-2\E F_{N}\leq\cP(\lambda_{*},q)-2P_{I}(\mu)+o(1) <-c
\]
 for some $c>0$ and $\eps$ sufficiently small. 

By Gaussian concentration \prettyref{eq:conc-gauss-1}, this
implies that with probability at least $1-Ke^{-N/K}$
\[
F_{2,N}((q-\eps,q+\eps))-2F_{N}<-c.
\]
Taking
$N$ sufficiently large and then $\epsilon$ sufficiently small, shows
that 
\[
\prob\left(\frac{1}{N}\log\pi_{N}^{\tensor2}(R_{12}\in(q-\epsilon,q+\epsilon))<-c'\right)\geq1-Ke^{-N\epsilon/K}
\]
for some $c'$, where we again apply \prettyref{eq:conc-gauss-1}. The result then follows by taking complements and limits.
\end{proof}

\subsection{Proof of \prettyref{thm:hard-spin-main-thm}}\label{sec:hard-spin-main-thm-proof}

We finally turn to the proof of {{\prettyref{thm:hard-spin-main-thm}}}.
In the following proof, $K>0$ will denote a constant that depends
at most on $\xi,h$ and possibly varies from line to line.

\begin{proof}[\textbf{\emph{Proof of \prettyref{thm:hard-spin-main-thm}}}]
Suppose that \NEV~ holds. Then there are two points in the support
of the overlap distribution, $\zeta$, that are also in
the support of the Parisi measure, $\pmeas$. Call these two points $q_1$ and $q_3$. 
Without loss of generality, $0\leq q_{1} < q_{3} \leq 1$. Furthermore at least
one of these points satisfy $\Lambda_{R}(q,\pmeas)>0$.

We begin by showing \prettyref{eq:exponential-rarity}.
Suppose first that $\Lambda_{R}(q_{3},\pmeas)>0$. Then by \prettyref{thm:exp-rare},
there is a $q_{2}$ with $q_{1}<q_{2}<q_{3}$ and $c,\eps$ such that 
\[
\frac{1}{N}\log\pi_{N}^{\tensor2}\left(R_{12}\in(q_{2}-\epsilon,q_{2}+\epsilon)\right)<-c,
\]
with probability $1-Ke^{-N/K}$. Furthermore, we may take $\epsilon$ sufficiently
small that 
\[
\epsilon<\frac{1}{4}\min\left\{ q_{2}-q_{1},q_{3}-q_{2}\right\} .
\]
This yields \prettyref{eq:exponential-rarity}.  
The case $\Lambda_R(q_1,\pmeas)>0$ is the same by symmetry.

We now show \prettyref{eq:separation}.
Since $q_{1},q_{3}\in\supp(\zeta)$, we have that 
\[
\zeta(q_{1}-\epsilon,q_{1}+\epsilon), \zeta(q_{3}-\epsilon,q_{3}+\epsilon)>0.
\]
Since $\zeta$ is by assumption the unique limit point of $\E \pi_N^{\tensor 2}(R_{12}\in\cdot)$ 
and the sets $(q_{i}-\epsilon,q_{i}+\epsilon)$ are relatively open, 
\[
\liminf\E\pi_{N}^{\tensor2}(q_{i}-\epsilon,q_{i}+\epsilon)>0
\]
by the portmanteau lemma for $i=1,3$, as desired. 
\end{proof}

\section{Free energy barriers in Spherical Models\label{sec:Overlap-Bounds-Sph}}
In this section, we prove \prettyref{thm:spherical-main-thm}.
As in the Ising spin setting, the main obstruction in proving the this result will be to show that 
 that certain regions of overlap values are exponentially rare as in \prettyref{eq:exponential-rarity}.
The arguments are analogous but technically simpler in this
spherical setting. 

The main result of this section is the following theorem. Recall that $\pmeas$ 
denotes the minimizer of \prettyref{eq:cs-func-1}.
\begin{thm}
\label{thm:exp-rare-manifold} Suppose that for some $q_{*}$ in the
support of $\mu$, $\Lambda(q_*,\mu)>0$. Then there is an $\epsilon_{0}$
such that for every $q$ in the punctured neighborhood $(q_{*}-\epsilon,q_{*}+\epsilon)\cap(0,1)\backslash\{q_{*}\}$,
there is an $\epsilon(q)$ and a $c(q)>0$ such that 
\[
\limsup_{N\to\infty}\frac{1}{N}\log\prob\left(\frac{1}{N}\log\pi_{N}^{\tensor2}(R_{12}\in(q-\epsilon,q+\epsilon))>-c\right)<0.
\]
If, furthermore, $q_{*}$ is in the support of $\zeta$ from \prettyref{eq:overlap-distn},
then it must be isolated.
\end{thm}
We remind the reader here of the Parisi-type formulation of the Crisanti-Sommers
formula. For $\nu\in\Pr([0,1])$ and $b\geq1$, define the spherical Parisi
functional, 
\[
P_{S}(\nu,b)=\begin{cases}
\frac{h^{2}}{b-\psi_{\nu}(0)}+\int_{0}^{1}\frac{\xi''(t)dt}{b-\psi_{\nu}(t)}+b-1-\log b-\int_{0}^{1}t\xi''(t)\nu(t)dt, & b-\psi_{\nu}(0)\geq0,\\
\infty & otherwise
\end{cases}
\]
where 
$\psi_{\nu}(t)=\int_{t}^{1}\xi''(s)\nu(s)ds.$
Let $\cA=\left\{ (\nu,b)\in\Pr([0,1])\times[1,\infty):b\geq\psi_{\nu}(0)\right\} $.
With these in hand, we also have the spherical Parisi formula is given
by 
\begin{equation}
2F=\min_{\nu,b}P_S(\nu,b).\label{eq:spherical-parisi}
\end{equation}
That these are equivalent was proved by Talagrand in \cite{TalSphPF06}.
We remind the reader of the following basic facts regarding the optimization
of this functional. Recall
$\varphi_{\nu}$ from \prettyref{eq:cs-param}. 
\begin{lem}
\label{lem:first-variation-consequences}$P_{S}(\nu,b')$ is strictly
convex and lower semicontinuous on $\Pr([0,1])\times[1,\infty)$ equipped
with the product topology, where $\Pr([0,1])$ is equipped with the
weak{*} topology. In particular, there
is a unique minimizing pair $(\mu,b)$. This pair satisfies:
\begin{align}
b & >\psi_{\mu}(0)\text{ and }b>1\label{eq:strict-b-psi}\\
q & =\int_{0}^{q}\frac{\xi''(t)}{(b-\psi_{\mu}(t))^{2}}dt+\frac{h^{2}}{(b-\psi_{\mu}(0))^{2}}\qquad\forall q\in\supp(\mu)\label{eq:q-optimality}\\
b & =\xi'(1)-\xi'(q_{EA})+\frac{1}{1-q_{EA}}\label{eq:b-optimality}\\
\varphi_{\mu}(q) & =\frac{1}{b-\psi_{\mu}(q)}\qquad\forall q\in\supp(\mu)\label{eq:phi-psi}\\
P(\mu,b) & =\cC(\mu)=\min_{\nu\in\Pr([0,1])}\cC(\nu)
\end{align}
In particular, $\mu$ is the minimizer of \prettyref{eq:cs-func-1}.
\end{lem}
\begin{rem}
This was proved under the assumption that $\mu$ is $k$-atomic in
\cite[Section 4.]{TalSphPF06}. One can perform a first variation
argument directly to $P_{S}$ and $\cC$ to obtain these equalities
for general $\mu$. For the reader's convenience we sketch this in
\prettyref{app:The-Spherical-Parisi}. 
\end{rem}
We now remind the reader here of the Guerra-Talagrand bound for spherical models, with the choice
of parameters \prettyref{eq:q-degnerate-def}, \prettyref{eq:A-degenerate-def},
and \prettyref{eq:nu-def}. Let $(\mu,b)$ be the optimizers of \prettyref{eq:spherical-parisi},
and define 
\begin{align*}
\cP(\lambda,q) & =\log\left(\frac{b^{2}}{b^{2}-\lambda^{2}}\right)+\int_{0}^{q}\frac{\xi''(t)}{b-\lambda-\psi_{\mu}(t)}dt\\
 & \qquad+\frac{1}{2}\int_{q}^{1}\xi''(t)\left(\frac{1}{b-\lambda-\psi_{\mu}(t)}+\frac{1}{b+\lambda-\psi_{\mu}(t)}\right)dt\\
 & \qquad\quad-\lambda q+b-1-\log b-\int_{0}^{1}t\xi''(t)\mu(t)dt
 -\frac{h^{2}}{b-\lambda-\psi_{\mu}(0)}
\end{align*}
where $\psi=\psi_{\mu}$ and where $b$ is taken to solve \prettyref{eq:b-optimality}.
Observe that 
\[
\cP(0,q,b)=2P_S(\mu,b).
\]
We then have the following analogue of Talagrand's 2D Guerra bound,  \prettyref{cor:degenerate-GT-bound},
for the spherical setting, which is from \cite{TalSphPF06}. See also \cite{PanchTal07,ChenJSP2015spherical}.
\begin{thm}
\label{thm:GT-bound-spherical} For $\lambda$ such that $b-\psi_{\mu}(0)+\abs{\lambda}>0$
and every $q\geq0$, we have that 
\begin{align}
\lim_{\epsilon\to0}\limsup_{N\to\infty}\E F_{2,N}((q-\epsilon,q+\epsilon)\cap[0,1]) & \leq\cP(\lambda,q).\label{eq:GPT-ub}
\end{align}
\end{thm}
With this in hand we then have the following theorem which is an analogue
of \prettyref{thm:hard-spin-barrier}. 
\begin{thm}
\label{thm:soft-spin-fe-barrier} Let $q_{*}\in\supp(\mu)$ be such
that $\Lambda_{R}(q_{*},\mu)>0$. Then there is an $\eps_0$
such that for all $q\in(q_{*}-\epsilon_{0},q_{*}+\epsilon_{0})\cap(0,1)$
with $q\neq q_*$, there is a $\lambda_*(q)$ satisfying 
\[
\cP(\lambda,q)<\cP(0,q_{*}).
\]
\end{thm}
\begin{proof}
Observe that by \prettyref{lem:first-variation-consequences}, 
\[
b-\psi_{\mu}(0)=b-\psi_{\mu}(q_{0})=\frac{1}{\varphi_{\mu}(q_{0})}>0
\]
where $q_{0}=\min\supp(\mu)$. Thus for $\lambda$ in a neighborhood
of $0$ we may apply \prettyref{thm:GT-bound-spherical}. Differentiating
first in $\lambda$ and setting $\lambda=0$, we see that
\[
\frac{\partial}{\partial\lambda}\cP\vert_{\lambda=0}=\int_{0}^{q}\frac{\xi''(t)}{\left(b-\psi_{\mu}(t)\right)^{2}}dt+\frac{h^{2}}{(b-\psi_{\mu}(0))^{2}}-q.
\]
Taking $q\in\supp(\mu)$, we see that this is zero by \prettyref{eq:q-optimality}.
Differentiating this expression in $q$, we see that for $q\in\supp(\mu)$,
\[
\frac{\partial}{\partial q}\frac{\partial}{\partial\lambda}\cP=\frac{\xi''(q)}{(b-\psi_{\mu}(q))^{2}}-1=\xi''(q)\varphi_{\mu}^{2}-1=-\varphi_{\mu}^{2}\Lambda_{R}(q,\mu),
\]
where the second equality follows by \prettyref{lem:first-variation-consequences}.
Taking $q=q_{*}$ implies that this is strictly negative. Observe
finally that 
\[
\abs{\partial_{\lambda}^{2}\cP}<\infty,
\]
in a neighborhood of $\lambda=0$ for all $q$. The result then follows
by \prettyref{lem:calculus-lemma}. 

Alternatively,  note here that $\cP$ is $C^2$, 
and that $\partial_q\cP(0,q_*)=\partial_q^2\cP(0,q_*)=0$, 
so that the Hessian is of the form \prettyref{eq:hess-p}. 
Thus it has a negative eigenpair as in \prettyref{eq:eigenpair}, with eigen-direction in the first quadrant. 
Thus, $\cP$ decreases for $(\lambda,q)$ in this direction. 
\end{proof}

Finally we note the following.
\begin{proof}[\textbf{\emph{Proof of \prettyref{thm:exp-rare-manifold}}}]
 The proof of this theorem is identical to that of \prettyref{thm:exp-rare}.
 In particular, it is an immediate consequence of \prettyref{thm:soft-spin-fe-barrier}
and \prettyref{eq:conc-gauss-1}.
\end{proof}

\begin{proof}[\textbf{\emph{Proof of \prettyref{thm:spherical-main-thm}}}]
The proof that this result follows from  \prettyref{thm:exp-rare-manifold}
is identical to the proof that \prettyref{thm:exp-rare} implies  \prettyref{thm:hard-spin-main-thm}
so it is omitted. 
\end{proof}

\section{Examples and Applications\label{sec:Examples-and-Applications}}

The main motivation for the above results is to understand the dynamics
of spin glass models. In this section, we will remind the reader of
the connection between spectral gaps and dynamics. We then discuss
regimes under which the assumptions \assA, \RSB, \textbf{G\RSB}, \NEV, and \textbf{G\NEV}~  are known
to hold. In particular, we prove  \prettyref{thm:ising-spin-unif-elliptic} and \prettyref{thm:spherical-unif-elliptic}.

\subsection{Spectral Gaps and Mixing}

Our interest in $\lambda_{1}$ is that it is a classical measure of
the time to equilibrium. For example, in the Ising spin setting, it
measures the rate of $L^{2}$-mixing of the continuous time Markov
chain induced by $Q$. That is, consider the semigroup $P_{t}$ with
infinitesimal generator $L=I-Q$. Since $\lambda_{1}$ is the first
nontrivial eigenvalue of $L$, we have the inequality 
\begin{equation}
\Var_{\pi_{N}}(P_{t}f)\leq e^{-\lambda_{1}t}\Var_{\pi_{N}}(f)\quad\forall f\in L^{2}(\pi_{N})\label{eq:gap-implies-decay}
\end{equation}
If the spectrum of $Q$ is non-negative, then for the discrete time
Markov chain $(\sigma(t))$ in $\Sigma_{N}$ induced by $Q$, one
has the $\lambda_{1}$ also measures both the $L^{2}$-mixing and
the Total variation mixing
\begin{lem}
If the spectrum of $Q$ is non-negative then, 
\[
Var_{\pi}(Q^{n}f)\leq\left(1-\lambda_{1}\right)^{2n}Var_{\pi}(f).
\]
Furthermore, if we denote the total variation mixing time by $t_{mix}$,
then there exists constants $c,C>0$ such that
\[
\prob\left(\frac{1}{N}\abs{\log t_{mix}-\log(\frac{1}{\lambda_{1}})}>\epsilon\right)\leq Ce^{-cN\epsilon^{2}}.
\]
\end{lem}
\begin{proof}
The first inequality as well as the inequality 
\[
\frac{1}{N}\abs{\log t_{mix}-\log\frac{1}{\lambda_{1}}}\leq\frac{1}{N}\abs{\log\log\pi_{N}}\leq\frac{1}{N}\log\left(\abs{\max H_{N}}+\abs{\min H_{N}}\right)
\]
are classical \cite{LPW}. That the second term is less than $\epsilon$
with probability $1-Ce^{-cN\epsilon^{2}}$ follows by \prettyref{eq:max-bound}
combined with Gaussian concentration.
\end{proof}
The assumption that $Q$ has non-negative spectrum is not particularly
stringent. It is common in the literature to circumvent this issue
by working with the ``Lazy'' version of the chain, i.e., $\tilde{Q}=\frac{1}{2}(I+Q)$
which only makes an $O(1)$ to the mixing. Alternatively note that,
as we are mainly interested in lower bonds on the mixing time
then if $\lambda_1<1$ which will be true in our applications, 
$(\lambda_1)^{-1}$ is still a good lower bound on the mixing time

In the spherical setting, $\lambda_{1}$ is of interest as it measures
the $L^{2}$ mixing through the inequality \prettyref{eq:gap-implies-decay}
as well, where here $P_{t}$ is the heat semigroup 
\[
P_{t}=e^{-t\cL_{H}}
\]
 induced by $\cL_{H}$ from \prettyref{eq:Langevin}. 

\subsection{Verifying \assA, \RSB, \textbf{G\RSB}, \NEV, and \textbf{G\NEV}~ for Ising spin models.}

In this section, we will discuss a family of models to which these
results apply. In particular, we aim to prove \prettyref{thm:ising-spin-unif-elliptic} 
To understand how this result holds, we explain here how each of the
assumptions can be shown to hold. The proof of this theorem is at
the end of the subsection. 
\vspace{1 em}

\noindent\textbf{Assumption \assA:} In general it is expected that $\zeta$ exists
and is in fact characterized by the following additional assumption 
\begin{defn}
\label{def:parisi-condition} We say that $(\xi,h)$ satisfies \parisicondition ~
if the push forward of the overlap distribution through the map $f(x)=\abs{x}$
is the minimizer of \prettyref{eq:parisi-formula-ising-minimizing}.
\end{defn}
If one could show \parisicondition~ for any limiting overlap distribution,
then Assumption \textbf{A }would be an immediate consequence of the
strict convexity of the Parisi functional for even models. A class of models which
are known to satisfy \parisicondition ~ are the even generic models. 
\vspace{1 em}

\noindent\textbf{\RSB:} Most results regarding when \RSB~ hold currently focus
on the challenging analytical question of the phase diagram for the
Parisi measure. See \cite{AuffChen13,JagTobPD15,TalPM06} for
results in this direction.
\vspace{1 em}

\noindent\textbf{\NEV: }
It is known \cite{JagTobPD15,Chen15} that for the optimizer, $\pmeas$,
of \prettyref{eq:parisi-formula-ising-minimizing}, we have the following:
for every $q\in\supp(\pmeas)$, 
\begin{equation}
\begin{cases}
\E_{h}\left(\partial_x \phi_{\pmeas}\right)^{2}(q,X_{q}) & =q\\
\Lambda_{R}(q,\pmeas) & \geq0
\end{cases}.\label{eq:optimality-conditions}
\end{equation}
In our application, we are most interested in the case where $\pmeas$
is not an atom and the above inequality is strict.

The following result is a collection of several results already appearing
in the literature. 
\begin{lem}
\label{lem:ising-spin-check-conditions}
Suppose that 
$\xi=\beta^{2}\xi_{0}$  is convex and that $\xi''_0(0)=0$.
Then there is an $h_0$ such that for $h\leq h_0$.
\begin{itemize}
\item $q_0= \min\supp(\pmeas)$ satisfies $\Lambda_{R}(q_0,\pmeas)>0$. 
\item For $\beta$ sufficiently large, $\pmeas$ is not an atom.
\item If $\pmeas$ is not an atom, then \textbf{G}\RSB~, and \textbf{G}\NEV ~ hold.
\end{itemize}
 Suppose furthermore that $\xi$ is either generic or even generic.
 Then \parisicondition~ and Assumption \assA~ hold, and in particular, if $\pmeas$
 is not an atom then \RSB~ and \NEV~ hold.
 \end{lem}
\begin{proof}
That \parisicondition~ holds for $\xi_{0}$ generic and even generic is shown in \cite{chatterjee2009ghirlanda,Panch10,PanchSKBook}. That \parisicondition~ implies
Assumption \assA~ for even generic models at zero external field is clear by symmetry.
For generic models and for even generic models with non-zero external field this follows by
the positivity of the overlap distribution which is well-known:
for generic models this follows by Talagrand's positivity principle \cite{Tal03pos,Panch07pos},
and for even generic models with nonzero external field this follows by \cite[Theorem 7]{Chen15}.

It remains to prove that there is such an $h_0$. Suppose first that $h=0$.
In this case, it was shown in \cite{AuffChen13} that $q_0 = 0$. 
It is also known for all $\xi_{0}$ \cite{AuffChen13,auffinger2017sk}
that if $\beta$ is sufficiently large, $\pmeas$ is not a single
atom. \textbf{G}\RSB~ then immediately follows. To see that $\Lambda_{R}(0,\pmeas)>0$,
recall from \cite[Lemma 16]{JagTobSC15} that for all $\nu$, $0<\Delta \phi<1$.
Thus
\[
\Lambda_{R}(0)>1-\xi''(0)>0.
\]
Thus \textbf{G}\NEV ~ holds.
For $h>0$ we now argue by continuity. Observe that $P_I(\nu;h)$ is jointly continuous
in the pair $(\nu,h)$.  (Here we have made the dependence of $P_I$ on $h$ explicit.)
To see this, metrize the weak-* topology on $\Pr([0,1])$ with $d(\mu,\nu)=\int\abs{\mu([0,t])-\nu([0,t])}dt$.
It is well-known \cite{Guer01,JagTobSC15} that 
\[
\norm{\phi_\mu-\phi_\nu}\leq d(\mu,\nu),
\]
and that 
\[
\abs{\partial_x\phi_\mu}\leq 1.
\]
Thus $P_I$ is Lipschitz in the usual product metric. 
By a standard argument (e.g., the fundamental theorem of $\Gamma$-convergence)
$\mu_h$, the optimizer of $P_I(\cdot;h)$, is continuous in $h$. (Recall again that 
the minimizer is unique.) 

Observe furthermore that the map $q_0:\Pr([0,1])\to \R_+$ defined by
\[
q_0(\mu)=\min \supp(\mu)
\]
is upper semicontinuous in the weak-* topology. Thus 
\[
\Lambda_R(q_0,\mu_h)\geq 1- \xi''(q_0(h))>0
\]
for $h$ sufficiently small. We used here that $0<\Delta\phi<1$ for all $\nu$ and $\xi$ \cite{JagTobSC15}. 
Thus the  first point holds. The second is then immediate. The final point holds by continuity. 
The results regarding \NEV and \RSB~ for generic and even generic models are then immediate.
\end{proof}

\begin{proof}[\textbf{\emph{Proof of \prettyref{thm:ising-spin-unif-elliptic}}}]
 This follows by applying \prettyref{thm:hard-spin-main-thm} and
\prettyref{lem:ising-spin-check-conditions}.
\end{proof}

\subsection{Verifying \assA, \RSB, and \NEV ~ for Spherical models.}

In this section, we will discuss a family of spherical models to which
these results apply. In particular, we aim to prove \prettyref{thm:spherical-unif-elliptic} 
To understand how this result holds, we explain here how each of the
assumptions can be shown to hold. The proof of this theorem is at
the end of the subsection. Many of these results are common with the
Ising spin setting.
\vspace{1 em}

\noindent\textbf{Assumption \assA:}
Define \parisicondition~  as in \prettyref{def:parisi-condition}
except for $\zeta$ corresponding to a spherical model. As with Ising
spin models, \parisicondition~ holds for generic models. It was also
shown in \cite{PanchTal07}, that \parisicondition~ holds for the
so called \emph{Pure p-spin }models i.e., models of the form $\xi(t)=\beta^{2}t^{p}$
$p\geq4$ and even.
\vspace{1 em}

\noindent\textbf{\RSB:}
As in the Ising spin setting, most of the analysis regarding \RSB ~
concerns the phase diagram for the Parisi measure, $\pmeas$. In
the spherical setting, far more is known about the Parisi measure
$\pmeas$ \cite{AuffChen13,JagTob16boundingRSB,TalSphPF06}. It is
also known that Pure $p$-spin models\textbf{ }are 1RSB \cite{TalSphPF06}.
An explicit, finite dimensional characterization of the space in which
$\pmeas$ lives for general $\xi$ is described in\cite{JagTob16boundingRSB}.
In particular, one can numerical check the class of ansatzes provided
there. 
\vspace{1 em}

\noindent\textbf{\NEV:}
For $\pmeas$, we have the following two relations \cite{JagTob16boundingRSB,TalSphPF06}
for every $q\in\supp(\pmeas)$:
\[
\begin{cases}
-h^{2}+\int_{0}^{q}\frac{1}{\phi_{\pmeas}^{2}(s)}ds & =\xi'(q)\\
\Lambda_{R}(q,\pmeas) & \geq0
\end{cases}.
\]
We are interested in understanding when $\Lambda_{R}(q,\pmeas)$ is
strictly positive.
\vspace{1 em}

Consider the following result which collects results already appearing
in the literature.
\begin{lem}
\label{lem:spherical-check-conditions} Suppose that either:
\begin{enumerate}
\item $\xi=\beta^{2}\xi_{0}$ with $\xi''_{0}(0)=0$ is convex and  generic  or even generic, or
\item $\xi(t)=\beta^{2}t^{p}$ for some even $p\geq4$ .
\end{enumerate}
Then \parisicondition~ and Assumption \assA~hold.  Furthermore, there is an $h_0$ such that 
for $h\leq h_0$,
\begin{itemize}
\item $q_0=\min\supp(\pmeas)$ satisfies $\Lambda_{R}(q_0,\pmeas)>0$ . 
\item If $\pmeas$ is not an atom, then \RSB~ and \NEV~  hold
\item For $\beta$ sufficiently large, $\pmeas$ is not an atom.
\end{itemize}
\end{lem}
\begin{proof}
As in \prettyref{lem:ising-spin-check-conditions}, 
that \parisicondition~ holds and implies Assumption \assA~
in our setting is well-known by the same argument for
generic and even generic models. 
The only points to note are that: the differentiation argument
provided there holds using using the
differentiability of the Crisanti-Sommers formula, which follows by
an application of an envelope-type theorem as in \cite[Theorem 1.2]{TalSphPF06} 
and holds even if $h=0$. 
In the case of even generic models when $h\neq0$, use \cite[Theorem 7.2]{TalSphPF06} to
enforce positivity. That it holds for Pure
$p$-spin models was proved by \cite[Theorem 4]{PanchTal07}.
There the authors prove that the support of the Parisi measure 
and the absolute overlap distribution coincide 
to check that they are the same, note that by the same differentiation 
argument, the $p$-th moment of these two measures coincide. 

It remains to prove the existence of $h_0$. Suppose first  that $h=0$.
It is known \cite[Corollary 1.3]{JagTob16boundingRSB} that $0\in\supp(\pmeas)$ 
since $\xi(t)\neq\beta^{2}t^{2}$ for some
$\beta\leq1$ by assumption. Recall that by \cite[Proposition 2.3]{TalSphPF06},
$\pmeas$ is a single atom if and only if for every $s\in(0,1)$,
\[
\beta^{2}\xi_{0}(s)+\log(1-s)+s<0,
\]
which is evidently violated for 
\[
\beta>\sqrt{\frac{\log(2)-\frac{1}{2}}{\xi_{0}(1/2)}}
\]
by taking $s=1/2$. Thus for all $\beta$ sufficiently large $\pmeas$
is not a single atom. Thus \RSB~  immediately follows. Finally, 
\[
\Lambda_{R}(0,\pmeas)>0,
\]
so that \NEV~ holds. The result for $h>0$ then holds by continuity as
before after noting that $P_S$ is jointly lower semicontinuous in $(\nu,b,h)$
and continuous in $h$ for $(\nu,b)$ fixed. 
\end{proof}
\begin{proof}[\textbf{\emph{Proof of \prettyref{thm:spherical-unif-elliptic}}}]
 This follows by applying \prettyref{thm:spherical-main-thm} and
\prettyref{lem:spherical-check-conditions}.
\end{proof}

\section{Proof of results from \prettyref{sec:LDP}}\label{sec:LDP-proofs}
In this section we collect the proofs of the results from \prettyref{sec:LDP}.
We provide the proofs in the order that the corresponding theorems are stated in
the introduction.

\begin{proof}[\textbf{\emph{Proof of \prettyref{thm:LDP}}}]
By the Guerra-Talagrand upper bound, \cite[Lemma 2]{Panch15},
we have that for every $q\in [-1,1]$
\begin{equation}\label{eq:guerra}
\lim_{\eps\to0}\limsup_{N\to\infty}\frac{1}{N}\E\log\sQ_N(B_{\epsilon}(q)) \leq - I(q).
\end{equation}
By  Panchenko's lower bound \cite[Theorem 2]{Panch15}, we have that for every $q\in[-1,1]$
\begin{equation}\label{eq:dmitry}
\lim_{\eps\to0}\liminf_{N\to\infty}\frac{1}{N}\E \log \sQ_N(B_{\eps}(q)) \geq -I(q).
\end{equation}
By Gaussian concentration, \prettyref{eq:conc-gauss-1}, we then obtain,
\begin{align*}
\lim_{\epsilon\to0}\limsup_{N\to\infty}\frac{1}{N}\log \sQ_{N}(B_{\epsilon}(q)) & \leq -I(q)\\
\lim_{\epsilon\to0}\liminf_{N\to\infty}\frac{1}{N}\log \sQ_{N}(B_{\epsilon}(q)) & \geq -I(q),
\end{align*}
almost surely. The proof then follows by the following very elementary result from large deviations theory, whose proof
is left to the reader.

\begin{lem}\label{lem:LDP-modification}
Let $\cX$  be a Polish metric space. Let $\{P_{N}\}$ be a sequence
of Borel probability measures on $\cX$. Let $J:\cX\to[0,\infty]$
be a measurable function such that for every $x\in\cX$,
\begin{align*}
\lim_{\epsilon\to0}\limsup_{N\to\infty}\frac{1}{N}\log P_{N}(B_{\epsilon}(x)) & \leq -J(x)\\
\lim_{\epsilon\to0}\liminf_{N\to\infty}\frac{1}{N}\log P_{N}(B_{\epsilon}(x)) & \geq -J(x).
\end{align*}
Then $J$ is rate function and $P_{N}$ satisfies a weak large deviation principle (LDP) with rate
$N$ and  rate function $J$. 
\end{lem}

Indeed, by \prettyref{lem:LDP-modification}, $I$ is a rate function, and 
$\sQ_N$ almost surely has a 
weak LDP with rate function $I$ and rate $N$. Since $\cX=[-1,1]$ is compact, this is in fact
an LDP and $I$ is good. 
\end{proof}

Let us now prove \prettyref{prop:generalized-feb-to-feb}.
\begin{proof}[\textbf{\emph{Proof of \prettyref{prop:generalized-feb-to-feb}}}]
Suppose that \FEB~holds for some $q_1<q_2<q_3$ and $\eps>0$.
 By \prettyref{lem:witness-lemma}, \prettyref{eq:separation} 
implies that 
\[
\liminf\frac{1}{N}\log \sQ_N(R_{12}\in(q_i-\eps,q_i+\eps))=0
\]
for $i=1,3$. Since $I$ is the rate function of the quenched LDP for $R_{12}$ by \prettyref{thm:LDP}, 
we have that $I(\tilde{q}_1)=I(\tilde{q}_3)=0$ for some $\tilde{q}_i$ in these respective neighborhoods, by
the LDP upper bound.
Furthermore, by \prettyref{eq:exponential-rarity}, and since $I$ is the rate function,
\[
-I(q_2)<\limsup_{N\to\infty}\frac{1}{N}\log \sQ_N(R_{12}\in(q_2-\eps,q_2+\eps))<-C
\]
by the LDP lowerbound. Thus
\[
\cH\geq I(q_2)>C
\]
so that Generalized \FEB~holds.
\end{proof}

\begin{proof}[\textbf{\emph{Proof of \prettyref{thm:generalized-FEB-EVbound}}}]
Observe that by definition of the difficulty,
\[
\liminf_{N\to\infty}\frac{1}{N} \cD_\eps(\pi_N^{\tensor2},\frac{\eps}{4})
\geq \liminf_{N\to\infty}\frac{1}{N} \cD_\eps(R_{12})\geq\liminf_{N\to\infty} \frac{1}{N}\Phi(q_1,q_2,q_3;\eps)
\]
for every $q_1<q_2<q_3\in\cR_\eps$ almost surely. By \prettyref{thm:LDP},
\[
-\inf_{x\in(q-\eps,q+\eps)}I(x)\leq \liminf_{N\to\infty}\frac{1}{N}S(q;\eps)\leq \limsup_{N\to\infty}\frac{1}{N}S(q;\eps)\leq -\inf_{x\in [q-\eps,q+\eps]} I(x),
\]
almost surely. Combining this with the above, using that $I$ is lower semicontinuous, and taking suprema,
we obtain
\begin{equation}\label{eq:H-difficulty-lower-bound}
\liminf_{\eps\to0}\liminf_{N\to\infty}\frac{1}{N}\cD_\eps(\pi_N^{\tensor 2},\frac{\eps}{4})
\geq \cH,
\end{equation}
almost surely.
By \prettyref{thm:difficulty-thm-hypercube-new-version} and \prettyref{eq:H-difficulty-lower-bound}, we see that 
\[
\liminf \frac{1}{N}\log\lambda_1 \leq - \cH
\]
almost surely, yielding the desired inequality. The result is then immediate by definition.
\end{proof}

\begin{proof}[\textbf{\emph{Proof of \prettyref{prop:equivalence-GFEB}}}]
Suppose that Generalized \FEB~holds. Then there is some $q_1<q_2<q_3$ such that 
\[
I(q_2)>I(q_1)+I(q_3).
\]
Consequently, either there is some $\tilde{q}_1\in [-1,q_2]$  with $I(\tilde{q_1})=0$
or there is some $\tilde{q}_3\in [q_2,1]$ such that $I(\tilde{q}_3)=0$.
In the first case, $I$ is not monotone to the right of $\tilde{q}_1$. 
In the second case, $I$ is not monotone to the left of $\tilde{q}_2$. 
Thus Generalized \FEB~is a sufficient condition for this monotonicity.
That it is necessary is immediate by definition.
\end{proof}

\begin{proof}[\textbf{\emph{Proof of \prettyref{thm:NREV-and-rate-function}}}]
This is simply a restatement of  \prettyref{thm:hard-spin-barrier}.
\end{proof}

Finally we have the following lemma which is an immediate consequence of \prettyref{lem:witness-lemma} and 
the large deviation principle, \prettyref{thm:LDP}.

\begin{lem}\label{lem:witness-lemma-RF}
Fix $q\in[-1,1]$. Suppose that for every $\eps>0$, 
\[
\liminf \E \sQ_N((q-\eps,q+\eps))>0.
\]
Then $I(q)=0$.
\end{lem}

\subsection{Regularity of the rate function}

Before turning to the proof or \prettyref{thm:gfeb-main-thm},
we note here the following regularity results regarding the
rate function $I.$
\begin{lem}\label{lem:I-reg}
For every $\xi$ convex, $I$ is continuous. In particular, it is (1/2)--H\"older.
\end{lem}
\begin{proof}
Without loss of generality we may take $h=0$. The case $h>0$ is identical.
We begin by showing that if $\eta<\epsilon$, then there is a $C=C(\xi)$ such that
\begin{equation}\label{eq:rest-fe-cont}
\abs{\E F_{2,N}(u-\epsilon,u+\epsilon)-\E F_{2,N}(u-\eta,u+\eta)}\leq J(\epsilon-\eta)+C\cdot(\epsilon-\eta)
\end{equation}
where 
\[
J(x)=-x\log x-(1-x)\log(1-x).
\]
To this end, for each $\sigma\in\Sigma_{N},$ let 
\[
B_{u,\epsilon}(\sigma)=\left\{ \sigma':R(\sigma,\sigma')\in(u-\epsilon,u+\epsilon)\right\} .
\]
Recall that, 
\[
\frac{1}{N}\log\abs{B_{1,\epsilon}(\sigma)}\leq J(\epsilon/2).
\]
Fix $\epsilon>\eta>0$ and let $\pi^{\sigma}:B_{u,\epsilon}(\sigma)\to B_{u,\eta}(\sigma)$
be the map that takes $\sigma'$ to $\pi(\sigma')\in B_{u,\eta}(\sigma)$
such that the Euclidean distance, $d(\pi(\sigma'),\sigma')$, is minimal. As $\Sigma_{N}$
is finite, this map is well-defined. Furthermore, we can choose $\pi(\sigma')$
so that $d(\pi(\sigma'),\sigma'))\leq2\sqrt{N}(\epsilon-\eta)$. 

Let $A_{\sqrt{N}}$ denote the ball in $\R^{N}$ of radius $\sqrt{N}$.
By Dudley's entropy bound \cite{LedouxTalagrand}, for any $\delta>0$, then
\begin{align*}
\E\sup_{\substack{d(\sigma^{1},\sigma^{2})\leq\delta\sqrt{N}\\
\sigma^{1},\sigma^{2}\in A_{N}^{\times2}
}
}\abs{H(\sigma^{1})-H_{N}(\sigma^{2})} & \lesssim_{\xi}N\delta.
\end{align*}
Combining these estimates yields 
\begin{align*}
\E F_{N}(u-\epsilon,u+\epsilon) & =\frac{1}{N}\E\log\int_{\Sigma_{N}}\int_{\sigma^{2}\in B_{u,\epsilon}(\sigma^{1})}e^{H(\sigma^{1})+H(\sigma^{2})}d\sigma^{2}d\sigma^{1}\\
 & \leq\frac{1}{N}\E\log\int_{\Sigma_{N}}\int_{\sigma^{2}\in B_{u,\epsilon}(\sigma^{1})}e^{H(\sigma^{1})+H(\pi^{\sigma^{1}}(\sigma^{2}))}d\sigma^{2}d\sigma^{1}+C\cdot(\epsilon-\eta)\\
 & \leq\frac{1}{N}\E\log\int_{\Sigma_{N}}\int_{\sigma^{2}\in B_{u,\eta}(\sigma^{1})}e^{H(\sigma^{1})+H(\sigma^{2})}\abs{B_{1,2(\epsilon-\eta)}(\sigma^{2})}d\sigma^{2}d\sigma^{1}+C\cdot(\epsilon-\eta)\\
 & = \E F_{N}(u-\eta,u+\eta)+C\cdot (\epsilon-\eta)+J(\epsilon-\eta).
\end{align*}
Since 
\[
J(\epsilon-\eta)\leq K\sqrt{\epsilon-\eta}
\]
for $\epsilon$ sufficiently small and $K>0$, we obtain the desired inequality. 

If $u,v$ are such that $\abs{u-v}=\delta$, then for any $\eta<\delta<1/2$,
the above yields, 
\[
\E F_{N}(u-\eta,u+\eta)-F_{N}(v-\eta,v+\eta)\leq\E F_{N}(v-2\delta,v+2\delta)-F_{N}(v-\eta,v+\eta)\lesssim_{\xi}\sqrt{2\delta-\eta}.
\]
Thus by symmetry
\[
\abs{\E F_{N}(u-\eta,u+\eta)-\E F_{N}(v-\eta,v+\eta)}\lesssim_{\xi}\sqrt{2\delta-\eta}.
\]
Combining the above bounds with Guerra--Talagrand and Panchenko's bounds \eqref{eq:guerra}-\eqref{eq:dmitry}, if
we send $N\to\infty$ and then $\eta\to0$, we obtain 
\[
\abs{I(u)-I(v)}\lesssim_\xi \sqrt{\delta}
\]
 as desired. 
\end{proof}
\begin{lem}\label{lem:exp-cont}
 Suppose that $\xi$ is convex. The map $(\beta,h,\xi)\mapsto I_{\beta,h,\xi}$
is strongly continuous from $\R_{+}^2\times C([-1,1])$ to $C([-1,1])$.
\end{lem}
\begin{proof}
Again by \eqref{eq:guerra}-\eqref{eq:dmitry}, it suffices to
show that for $(\beta_{1},h_1,\xi_{1}),(\beta_{2},h_2,\xi_{2})$, and for
every $u\in[-1,1]$ and $\epsilon>0$ sufficiently small, 
\begin{equation}\label{eq:FE-cty-xi}
\abs{\E F_{2,N}^{\beta_{1},h_1,\xi_{1}}(u-\epsilon,u+\epsilon)-\E F_{2,N}^{\beta_{2},h_2,\xi_{2}}(u-\epsilon,u+\epsilon)}\leq\norm{\beta_{1}^{2}\xi_{1}-\beta_{2}^{2}\xi_{2}}_{\infty}+\abs{h_1-h_2},
\end{equation}
for some universal $c$. Furthermore, by Jensen's inequality, it suffices to take the case $h_1=h_2$. 
This case follows by a standard interpolation estimate. Indeed, fix such a $u$ and $\epsilon$ and let $H_{1}$ denote
the Hamiltonian corresponding to $\beta_{1}^{2}\xi_{1}$ and $H_{2}$
that corresponding to $\beta_{2}^{2}\xi_{2}$. Then, if we define the interpolating Hamiltonian
\[
H_{t}(\sigma)=\sqrt{t}H_{1}(\sigma)+\sqrt{1-t}H_{2}(\sigma),
\]
and let 
\[
\phi(t)=\E F_{2,N}^{t}(u-\epsilon,u+\epsilon).
\]
Gaussian integration by parts, see, e.g., \cite[Lemma 1]{PanchSKBook}, implies that
\[
\phi'(t) = \E\int C(\sigma^1,\sigma^1)- C(\sigma^1,\sigma^2)d\pi_t^{\tensor 2}
\]
where $C(\sigma^1,\sigma^2) = \beta^2_1\xi_1(R_{12})-\beta_2^2\xi_2(R_{12})$
and $\pi_t$ is the Gibbs measure corresonding to $H_t$. The inequality \eqref{eq:FE-cty-xi} is the immediate.
\end{proof}
\begin{thm}\label{thm:zeros}
Suppose that $\xi$ is convex. We have that 
\[
\supp(\mu)\subset\left\{ I=0\right\}. 
\]
\end{thm}
\begin{proof}
Suppose first that $\xi$ is even generic. Since \textbf{P} and \textbf{A}~ hold by \prettyref{lem:ising-spin-check-conditions},
\[
\supp(\mu)\subset \supp(\zeta),
\]
from which the result follows by \prettyref{lem:witness-lemma-RF}.
Suppose now that $\xi$ is only
even. Let $\xi_{\epsilon}=\xi+\epsilon\eta$ where $\epsilon=\sum\frac{1}{2^{p}}t^{p}$.
Then $\xi_{\epsilon}$ is even generic. Thus 
\[
\supp(\mu_{\epsilon})\subset\left\{ I_{\epsilon}=0\right\} .
\]
It is well-known \cite{JagTobSC15} that $\mu_{\epsilon}\to\mu$ weakly. Recall
the following basic fact.
\begin{lem}
If $\nu_{\epsilon},\nu\in\Pr([0,1])$ and $\nu_{\epsilon}\to\nu$
weakly as measures. Then for every $q\in\supp(\nu)$ there is a sequence
$q_{\epsilon}\to q$ with $q_{\epsilon}\in\supp(\nu_{\epsilon})$. 
\end{lem}
Thus for $q$ in $\supp(\mu)$ if we take $(q_\eps)$ as above,  \prettyref{lem:I-reg} and \prettyref{lem:exp-cont} yield
\begin{align*}
0 & \leq I(q)\leq\liminf_{\epsilon\to0}I(q_{\epsilon})\leq\liminf I(q_{\epsilon})-I_{\epsilon}(q_{\epsilon})+I_{\epsilon}(q_{\epsilon})\leq\liminf_{\epsilon\to0}\abs{I(q_{\epsilon})-I_{\epsilon}(q_{\epsilon})}\leq\liminf\epsilon\norm{\eta}_{\infty}=0
\end{align*}
as desired.
The case $\xi$ is convex is dealt with analogously, by adding a nonzero external field and sending $h\to0$. The only difference is to note that, since
the perturbed model $\xi +\eps \eta$ is such that the collection $\{t^p:\beta_p\neq0\}\cup\{1\}$ is total in $C([0,1])$,
 conditions \parisicondition~ and \assA~ still hold when $h > 0$ by the same argument 
from \prettyref{lem:ising-spin-check-conditions}.
\end{proof}
With this in hand, we may now prove \prettyref{thm:gfeb-main-thm}.
\begin{proof}[\textbf{\emph{Proof of \prettyref{thm:gfeb-main-thm}}}]
By \textbf{GRSB}, there are at least two points in the support of $\mu$
call them $q_{1}<q_{3}$. By \prettyref{thm:zeros}, 
\[
I(q_{1})=I(q_{3})=0.
\]
By \textbf{GPREV}, at least one of these points has positive
\repev. Thus $I$ is positive in a punctured neighborhood of this point by \prettyref{thm:NREV-and-rate-function}. 
In particular, there is a point
$q_{2}$ between $q_{1}$ and $q_{3}$ for which $I(q_{2})>0$. Thus
$\cH$>0, that is, \textbf{GFEB} holds.
\end{proof}

\subsection{Applications of GFEB}
Let us now turn to the proof of our main examples.

First we have the following.
\begin{proof}[\textbf{\emph{Proof of \prettyref{thm:gfeb-examps}}}]
By \prettyref{lem:ising-spin-check-conditions},
 \textbf{GPREV} holds for these models. 
The first result then follows by \prettyref{thm:gfeb-main-thm}. 
The remaining follows by the second point of \prettyref{lem:ising-spin-check-conditions}.
 \end{proof}

Let us now turn to the proof of  \prettyref{cor:Ts-T2}. Recall the following theorem of Auffinger--Chen \cite[Theorem 4]{AuffChen13} .
\begin{thm}\label{thm:ac-thm}
We have that $0\in\supp(\mu)$ for all $\beta$. Furthermore, if $q_{AC}$
denotes the solution of 
$\xi(q_{AC})=1$,
then $\mu([0,q_{AC}\wedge1))=\mu\left(\{0\}\right).$
\end{thm}
\prettyref{cor:Ts-T2} then follows by a straightforward continuity argument.
\begin{proof}[\textbf{\emph{Proof of \prettyref{cor:Ts-T2}}}]
By \prettyref{thm:ac-thm} and  \prettyref{thm:zeros}, 
we know that $I(0)=0$ for all $\beta>0$. For every $\beta>\beta_{s},$
we know that there is some $q_{3}>0$ such that $I(q_{3})=0$. Let
\[
q_{*}=\liminf_{\beta\downarrow\beta_{s}}q_{3}(\beta).
\]
By \prettyref{lem:I-reg} and \prettyref{lem:exp-cont}, $I(q_{*})=0$. Furthermore,
by \prettyref{thm:ac-thm}, $q_{*}\geq q_{AC}>0.$ Thus there are $q_{1},q_{3}$
such that $I(q_{i})=0$ for $\beta=\beta_s$. 

By \prettyref{thm:NREV-and-rate-function}, for every $\beta>0$, there is an $\epsilon_{0}(\beta)>0$
such that for all $q\in(-\epsilon_{0},\epsilon_{0})\backslash\{0\},$
\[
I(q)>0.
\]
Thus there is a $q_{2}\in(q_{1},q_{3})$ such that $I(q_{2})>0=I(q_{1})+I(q_{3})$
for $\beta=\beta_{s}.$ The result the follows
by \prettyref{lem:exp-cont} and the intermediate value theorem. 
\end{proof}

\appendix

\section{Analytical properties of the Parisi PDE\label{app:Analytical-Properties-of}}

In this section, we collect basic facts about the Parisi PDE and its multidimensional
analogues. Basic results regarding this functional are treated in many different fashions
and are scattered throughout the literature \cite{TalPM06, bovier2009aizenman, AuffChen13, JagTobSC15, Chen15}.
For a systematic review of the 1-dimensional setting see \cite{JagTobSC15}. 
For the sake of completeness and as we 
imagine it will be useful for future research, we state these results 
in a general setting. Most of these results follow from arguments already appearing in the literature, so
our presentation will be brief.
In the following, we say that a function $f$ on $\R^d$ has \emph{at most linear growth at infinity} if there
are constants $a,b$ such that $\abs{f(x)}\leq a\norm{x}+b$.

Consider the following Cauchy problem. Let $T>t_0\geq0$. 
Suppose that $A(t):[0,T]\to \PSD_d$ is a $d\times d$ positive semidefinite matrix that 
is strictly positive definite on $(t_0,T]$. Suppose furthermore that that there is a non-decreasing 
function $\alpha(s):(t_0,T]\to \R_+$  and a constant $\kappa$ such that 
\begin{equation}\label{eq:A-inv-estimate}
\kappa Id \geq A(t)\geq \alpha(t) Id
\end{equation}
Finally, let $\nu(t)\in L^\infty([0,T])$ and $g\in C^\infty$ with uniformly bounded derivatives.
Consider the Cauchy problem
\begin{equation}\label{eq:general-ppde}
\begin{cases}
\partial_{t}u+\frac{1}{2}\left(\left(A,D^{2}u\right)+\nu(t)\left(Du,ADu\right)\right)=0\\
u(T,x)=g(x).
\end{cases}
\end{equation}
We say that $u$ is a weak solution to \prettyref{eq:general-ppde} 
if  $u$ is continuous in 
space and time with 
essentially bounded weak spatial derivative $Du$, and solves
\begin{equation}\label{eq:general-ppde-weak}
\int_{t_0}^T\int_{\R^d}-u\partial_{t}\varphi+\frac{1}{2}\left(u\left(A,D^{2}\varphi\right)+\nu(t)\left(Du,ADu\right)\varphi\right)dx^{2}dt+\int_{\R^{d}}\varphi(T,x)g(x)dx
\end{equation}
for any test function $\varphi\in C^\infty([t_0,T)\times \R^d).$

\subsection{Existence, uniqueness, and regularity of weak solutions}
The following is a minor modification of \cite[Theorem 2]{JagTobSC15}. The arguments provided there extend to the higher dimensional setting and also extends, with minor modifications
to the setting where the initial data is only bounded and not also square-integrable.
 In particular the heat kernel estimates from \cite[Eq. (3)]{JagTobSC15}
still hold so the arguments there still hold in any dimension. 

\begin{lem}
\label{lem:regularity-u} Suppose that $\nu(t)\in L^\infty$ and $A(t)$ are as above. Let $g_{\lambda}$
be a one parameter family of functions that are smooth and have uniformly bounded derivatives.
Then there is a unique weak solution, $u_\lambda$, to \prettyref{eq:general-ppde} for each $\lambda$. 
Furthermore, we have the following:
\begin{enumerate}
\item $u_\lambda$ is continuous in space, smooth in time, with uniformly bounded
spatial derivatives.
\item $u_\lambda$ and its derivatives are once weakly differentiable in time with
$\partial_{t}\partial_{x_{i}}^{j}u\in L_{t,x}^{\infty}$. 
\item There are constants $K_{n}$ that depend at most on $A$ and $\lambda$
such that 
\[
\norm{D^{n}u}_{L_{t,x}^{\infty}}\leq K_{n}(A,\lambda)\quad\forall n\geq1.
\]
If the derivatives of $g$ in $\lambda$ are uniformly bounded in
$\lambda,$ then $K_{N}$ depends on $A$ alone.
\end{enumerate}
\end{lem}
\begin{rem}\label{rem:regularity-u-application}
When $d=1$, $A=\xi''(t)$, and $g=\log\cosh(x)$ we are studying \prettyref{eq:PPDE-IVP-ising}.
When $d=2$, $A$ is strictly positive definite on $[0,T]$ and $g=f_\lambda(x)$ we are
in the setting of \prettyref{eq:multidim-ppde}. When $d=2$, $t_0=q$, $A$ is as in 
\prettyref{eq:A-degenerate-def}, and $g=f_\lambda$ we are in the setting  used in 
\prettyref{lem:derivatives}. In this case, we note that we have the estimate on $A$ with 
$\alpha(s)=(\xi''(s))^{-1}$. 
\end{rem}

Observe that this result applies to $u$ from \prettyref{eq:degenerate-multid-pfunc}.
Furthermore, we note here that $v$ from \prettyref{eq:v-def-pde},
satisfies the same bounds by the same argument.
\begin{lem}
\label{lem:regularity-v} We have that $v$ from \prettyref{eq:v-def-pde}
exists and is unique. Furthermore, $v\in C_{t}C_{x}^{\infty}$,
$\partial_{t}\partial_{x_{i}}^{j}v\in L_{t,x}^{\infty}$, and 
there are constants $K_{n}$ that depend at most on $\xi$ such that
\[
\norm{\partial_{x}^{n}v}_{L_{t,x}^{\infty}}\leq K_{n}\quad\forall n\geq1
\]
\end{lem}
\begin{rem}
We note here again that $K_{n}$ does not depend on $\lambda$
\end{rem}

\subsection{Continuous and Differentiable dependence of the solution of the Parisi PDE}\label{app:Well-Posedness-of-the}

The proof of \prettyref{thm:hard-spin-barrier} requires differentiable dependence of the solution
of \eqref{eq:multidim-ppde} on $\lambda$ and $q$.

\subsubsection{Differentiable dependence in initial data}

We aim to differentiate the solution of \prettyref{eq:multidim-ppde}
in $\lambda$. This follows by classical differentiable dependence
arguments. This type of derivative has already been used many times
in the literature \cite{Chen15,chen2015fluctuations,chen2016energy,Panch16}.

More generally, we have the following result regarding the differentiable dependence
of the Parisi PDE on its initial data.

\begin{lem}
\label{lem:cts-depce-hard-spin} 
Let $g_\lambda$ be a one parameter family of functions in $C^\infty(\R^d)$ with uniformly
bounded derivatives. Suppose that the family of maps $\lambda\mapsto D^kg_\lambda(x)$ is  $K$-Lipschitz uniformly in $x\in \R^d$ for $k\in[n]$.
Let $u(\lambda)$ be the corresponding solutions to \prettyref{eq:general-ppde}.
with at most linear growth at infinity. 
The map $\lambda\mapsto(u(\lambda),Du(\lambda),D^{2}u(\lambda),\ldots,D^{n}u(\lambda))$
satisfies the Lipschitz property 
\[
\norm{u(\lambda)-u(\lambda')}_{C([t_{0},T];C^{n}(\R^{2}))}\lesssim_{n,\kappa,K}\abs{\lambda-\lambda'}.
\]

Suppose that $g_\lambda(x)$ is twice differentiable in $\lambda$ (pointwise in x) and that the derivatives in $\lambda$ are  Lipschitz 
in $\lambda$ uniformly in $x$.
Then $\partial_{\lambda}u$ and $\partial_\lambda^2 u$
exists for $t$ in $[t_{0},T]$ and are a mild solution to the heat
equations
\begin{align}
&\begin{cases}
\partial_{t}\partial_{\lambda}u+\frac{1}{2}\left(\left(A,D^{2}\partial_{\lambda}u\right)+2\nu\left(ADu,D\partial_{\lambda}u\right)\right)=0\\
\partial_{\lambda}u(T,x)=\partial_{\lambda}g.
\end{cases}\label{eq:diff-pde-lambda}\\
&\begin{cases}
\partial_{t}\partial_{\lambda}^{2}u+\frac{1}{2}\left(\left(A,D^{2}\partial_{\lambda}^{2}u\right)+2\nu\left(ADu,D\partial_{\lambda}^{2}u\right)\right)=-\nu\left(AD\partial_{\lambda}u,D\partial_{\lambda}u\right)\\
\partial_{\lambda}u(T,x)=\partial_{\lambda}^{2}g
\end{cases}.\label{eq:diff-pde-twice-lambda}
\end{align}
Furthermore, the map $\lambda\mapsto(\partial_{\lambda}u,D\partial_{\lambda}u,D^{2}\partial_{\lambda}u)$
is continuous as a map $\R\mapsto C([t_{0},T];C^{2}(\R^{2})).$
\end{lem}

With this result, we then also have the following result which is an immediate Corollary. 
Let $v$ be as in \prettyref{eq:v-def-pde}.
\begin{lem}
\label{lem:diff-depce-lambda-v}We have that $\partial_{\lambda}v$
exists for $t$ in $[0,T]$ and for $t\leq q$ , it is a mild solution
to 
\begin{equation}
\begin{cases}
\partial_{t}\partial_{\lambda}v+\frac{\xi''}{2}\left(\partial_{x}^{2}\partial_{\lambda}v+2\nu v_{x}\partial_{x}\partial_{\lambda}v\right)=0 & (t,x)\in[0,q]\times\R\\
\partial_{\lambda}v(\tau,x)=\partial_{\lambda}u(\tau,x) & \tau\geq q
\end{cases}.\label{eq:diff-pde-lambda-v}
\end{equation}
Furthermore $\partial_{\lambda}^{2}v$ exists for $t$ in $[0,T]$
and is a mild solution to 
\[
\begin{cases}
\partial_{t}\partial_{\lambda}^{2}v+\frac{\xi''}{2}\left(\partial_{x}^{2}\partial_{\lambda}^{2}v+2\nu v_{x}\partial_{x}\partial_{\lambda}^{2}v\right)=-\xi''(t)\nu\left(\partial_{x}v\right)^{2} & (t,x)\in[0,q]\times\R\\
\partial_{\lambda}v(\tau,x)=\partial_{\lambda}u(\tau,x,x) & \tau\geq q
\end{cases}.
\]
\end{lem}
That $\partial_{\lambda}v$ is well-defined for $\tau\geq q$ follows
immediately from the existence for $\partial_{\lambda}u$. In particular,
we may also write this as an inhomogeneous heat equation on $[0,T]\times\R^{d}$. 

\subsubsection{Lipschitz dependence on ellipticity. \label{app:Lipschitz-dependence-on-ellpiticity}}

To prove \prettyref{thm:hard-spin-main-thm}, it is useful to know
that $\cP(\lambda,q)$ from \prettyref{eq:cP-def} depends continuously
on $q$. To this end, take $q_{1},q_{2}\in[0,1],$ and let $v_{1}$
and $v_{2}$ be the corresponding solutions to \prettyref{eq:v-def-pde}.
Observe that $u_{\nu}$ in that definition may be taken to be the
weak solution of \prettyref{eq:multidim-ppde} , with $A=\xi''(t)Id$
and $\gamma=\mu$ for all $t\in[0,1]$ since we only evaluate $u_{\gamma}$
for $t\geq q_{i}$. Our goal is then to prove the following lemma.
Let $w=v_{1}-v_{2}.$
\begin{lem}
\label{lem:cts-depce-q-hard-spin} We have that 
\[
\abs{D^{n}w(0,x)}\lesssim_{\xi,n}\abs{q_{1}-q_{2}}.
\]
Furthermore, $\partial_{\lambda}v(0,x)$
and $\partial_{\lambda}^{2}v(0,x)$ have Lipschitz dependence in $q$
uniformly in $\lambda$ as well. 
\end{lem}

\subsection{Proof of {\prettyref{lem:derivatives}\label{subsec:Proof-of-deriv-lem}}}

We are now in the position to prove {{\prettyref{lem:derivatives}.}}
\begin{proof}[\textbf{\emph{Proof of \prettyref{lem:derivatives}}}]
The existence, uniqueness, and regularity of $u(\lambda)$ and $v(\lambda)$ follows from \prettyref{lem:regularity-u}-\ref{lem:regularity-v}.
The map 
\[
\lambda\mapsto f_{\lambda}(x)
\]
is Lipschitz from $\R\to f_{0}+C([0,T];C^{n}(\R^{2}))$. 
Indeed, $f$ is smooth in the pair $(\lambda,x)$, and $\partial^k_x\partial^l f_\lambda(x)$ 
is bounded by a constant that depends on $k$ and $l$ alone. 
The differentiability of these in $\lambda$  follows from \prettyref{lem:cts-depce-hard-spin}~
 and \prettyref{lem:diff-depce-lambda-v}. To see that the derivatives
satisfy the representation formula, note that $\partial_{\lambda}u$
and $\partial_{\lambda}v$ are bounded, smooth in space, and weakly
differentiable in time with bounded weak derivative by virtue of being
the unique solutions of the differentiated equations \prettyref{eq:diff-pde-lambda}
and \prettyref{eq:diff-pde-lambda-v} respectively, which are (time
inhomogeneous) heat equations with coefficients that are smooth, bounded,
Lipschitz in space, and bounded measurable in time. Thus we may apply
It\^o's lemma (see, e.g., \cite{stroock1979multidimensional}) after
observing that the infinitesimal generators of \prettyref{eq:tilde-X-multid}
and \prettyref{eq:tilde-X} are given by 
\begin{align*}
L_{1} & =\frac{1}{2}\left(\left(A,D^{2}\cdot\right)+2\nu\left(ADu,D\cdot\right)\right)\\
L_{2} & =\frac{\xi''}{2}\left(\partial_{x}^{2}+2\nu v_{x}\partial_{x}\right).
\end{align*}
The continuous dependence in $\lambda$ of the second derivative
follows by \prettyref{lem:cts-depce-hard-spin}.
The continuous dependence in $q$ is by \prettyref{lem:cts-depce-q-hard-spin}.

We note that for these two lemmas,
the continuity is uniform in $\lambda$ since they depend on $\lambda$
only through the derivative bounds on $u,\partial_{\lambda}u,\partial_{\lambda}^{2}u,v$,
and $\partial_{\lambda}v$, which are themselves uniform in $\lambda$
by  \prettyref{lem:regularity-u}.

\end{proof}

\section{Optimality Conditions for the spherical parisi functional\label{app:The-Spherical-Parisi}}

In this section, we briefly sketch the proof of \prettyref{lem:first-variation-consequences}.

\begin{proof}[\textbf{\emph{Sketch of Proof of \prettyref{lem:first-variation-consequences}}}]
That $P_{S}$ is jointly strictly convex follows from strict convexity
of each term. The lower semicontinuity follows from the fact that
that if $\mu_{n}\to\mu$ weak-{*} then $\psi_{\mu_{n}}\to\psi_{\mu}$
point-wise almost everywhere, combined with Fatou's lemma. 

Since $\psi_{\mu}\leq\xi'(1)-\xi'(t)$, we see that there is a $C(\beta,h),b_{0}(\beta,h$)
such that for $b\geq b_{0}$, 
\[
P(\mu,b)\leq C\cdot b\quad\forall\mu.
\]
Thus we may restrict the optimization to the compact set $\Pr([0,1])\times[1,C\cdot b_{0}].$
Thus the optimal pair exists and is unique.

We now characterize this optimal pair, call it $(\mu,b)$. Suppose $\gamma_{\theta}=(\mu_{\theta},b_{\theta})$
is a path $\gamma:[0,1]\to\cA$ such that $\gamma_{0}=0$, and such
that the right derivative 
\[
\frac{d}{dt}\vert_{t=0}P(\gamma_{t})
\]
exists. Then if $(\mu,b)$ is optimal, then
\begin{equation}
\frac{d}{d\theta}\vert_{\theta=0}P(\gamma_{\theta})\geq0.\label{eq:first-order-opt-P}
\end{equation}
We refer to this as the first order optimality condition. 

We begin with the first point. Applying \prettyref{eq:first-order-opt-P}
to the path $(1+\theta,\mu)$, we see that 
\[
\frac{d}{d\theta}\vert_{\theta=0}P(\gamma_{\theta})<0
\]
so that $b>1$. Suppose now that $b=\psi_{\mu}(0)$. Then 
\[
\frac{\xi''(t)}{b-\psi_{\mu}(t)}\geq\frac{1}{t}
\]
which is not integrable. Thus $b>\psi_{\mu}(0)$. 

Now for the second point. Take any $\nu$ and consider the variation
that sends $\theta\mapsto(b,\nu_{\theta})$, where $\nu_{\theta}=\theta\nu+(1-\theta)\mu$
mixes the two measures. For $\theta$ sufficiently small, this path
is admissible since since $b>\psi_{\mu}$. Since $\gamma\mapsto P(\gamma,b)$
is strictly convex in on the set 
\[
E(b)=\left\{ \nu\in\Pr([0,1]):b-\psi_{\nu}(0)\geq0\right\} ,
\]
we see that 
\[
\frac{d}{d\theta}^{+}P(\nu_{\theta},b)\geq0.
\]
Computing this derivative and re-arranging, we obtain that 
\[
\left\langle G(t),\nu-\mu\right\rangle \geq0
\]
where 
\[
G(t)=\int_{t}^{1}\xi''(s)\left(\frac{h^{2}}{(b-\psi_{\mu})^{2}(0)}+\int_{0}^{s}\left[\frac{\xi''(\tau)}{(b-\psi_{\mu})^{2}(\tau)}-1\right]d\tau\right)ds.
\]
Thus $\mu$ is optimal only if 
\[
\mu(G_{b}(t)=\min_{t\in[0,1]}G_{b}(t))=1.
\]
In particular $G'(q)=0$ for all $q\in\supp(\mu)$. This yields \prettyref{eq:q-optimality}. 

For the third point, since $b>\psi_{\mu}$, we see that we may take
a full variation in $b$ so that 
\[
\frac{d}{db}P(b,\mu)=0.
\]
This combined with \prettyref{eq:q-optimality} for $q=q_{EA}$ yields
\prettyref{eq:b-optimality}.

For the fourth point, we see that by integrating \prettyref{eq:q-optimality},
\[
\varphi_{\mu}(q')-\varphi(q)=\frac{1}{b-\psi_{\mu}(q')}-\frac{1}{b-\psi_{\mu}(q)}
\]
for all $q,q'\in\supp(\mu).$ Taking $q'=q_{EA}$ yields \prettyref{eq:phi-psi}.

For the finally point, an explicit computation yields 
\[
P(\mu,b)=\cC(\mu).
\]
The result then follows by \prettyref{eq:spherical-parisi}. 
\end{proof}

\bibliographystyle{plain}
\bibliography{sphericalmixinglowtemp}

\end{document}